\documentclass{article}
\usepackage{pdfsync}
\usepackage[plainpages=false]{hyperref}
\usepackage{amsfonts,amsmath,amssymb,amsthm}
\usepackage{latexsym,lscape,rawfonts,mathrsfs}
\usepackage{mathtools}
\usepackage[utf8]{inputenc}
\usepackage[overload]{empheq}
\usepackage{cases}

\usepackage{multicol}
\usepackage{xcolor}

\usepackage{tocloft}

\usepackage[all]{xy}
\usepackage{eufrak}
\usepackage{makeidx}
\usepackage{graphicx,psfrag}
\usepackage{pstool}
\usepackage{float}
\usepackage{tikz-cd}

\usepackage{array,tabularx}

\usepackage{setspace}

\usepackage{appendix}

\usepackage{txfonts}

\newcommand{\M}{\ensuremath{\mathcal{M}}}

\newcommand{\R}{\ensuremath{\mathbb{R}}}
\newcommand{\Z}{\ensuremath{\mathbb{Z}}}
\newcommand{\CP}{\ensuremath{\mathbb{CP}}}

\newcommand{\ba}{\begin{align*}}
\newcommand{\ea}{\end{align*}}

\newcommand{\tr}{\mathrm{tr}}%trace

\newcommand{\divg}{\mathrm{div}}
\newcommand{\Rm}{\mathrm{Rm}}
\newcommand{\Ric}{\mathrm{Ric}}
\newcommand{\s}{\mathrm{scal}}

\newcommand{\V}{\mathrm{Vol}}
\newcommand{\coker}{\mathrm{coker}}

  \newcommand{\SO}{\mathcal{O}}

\makeatletter
\def\ExtendSymbol#1#2#3#4#5{\ext@arrow 0099{\arrowfill@#1#2#3}{#4}{#5}}

\makeatother

\makeatletter
\def\ExtendSymbol#1#2#3#4#5{\ext@arrow 0099{\arrowfill@#1#2#3}{#4}{#5}}

\makeatother

\definecolor{orange}{rgb}{1,0.5,0}
\definecolor{brown}{rgb}{0.48,0.33,0.19}
\definecolor{miao}{cmyk}{0.5,0,0.2,0.2}
\definecolor{qiao}{gray}{0.96}

\newtheorem{prop}{Proposition}[section]
\newtheorem{proposition}[prop]{Proposition}

\newtheorem{theorem}[prop]{Theorem}

\newtheorem{lemma}[prop]{Lemma}

\newtheorem{corollary}[prop]{Corollary}

\newtheorem{remark}[prop]{Remark}

\newtheorem{definition}[prop]{Definition}

\newtheorem{notation}[prop]{Notation}

\numberwithin{equation}{section}

\newcommand{\gtd}{g_{t}}
\newcommand{\gtidi}{g_{t_i}}
\newcommand{\gtdp}{g'_{t}}
\newcommand{\rtd}{r_{t}}
\newcommand{\ctab}{C^{2,\alpha}_{\beta,\beta}}

\newcommand{\cab}{\rtd^{-2}C^{0,\alpha}_{\beta,\beta}}

\newcommand{\TO}{\tilde{\SO}}

\title{Resolution of compact Einstein orbifolds in general dimensions}
\author{Yichen Yao}
\date{}

\begin{document}
\maketitle
\begin{abstract}
    Given a noncollapsing sequence of $m$-dimensional compact Einstein manifolds with a uniform energy bound, the Gromov-Hausdorff limit is a compact Einstein orbifold with at most finitely many singularities. Conversely, starting with a compact Einstein orbifold, we are interested in whether there exists a sequence of smooth Einstein metrics converging to it. 
    
    In this paper, we provide a negative answer. We give an explicit obstruction, such that if an Einstein orbifold with negative scalar curvature appears as a noncollapsing limit of compact Einstein manifolds, then the obstruction must vanish. Such an obstruction links the curvature at the orbifold singularity and the geometry of the blow-up limit. As an example, the obstruction never vanishes for hyperbolic orbifolds, so they can not be approximated by smooth Einstein manifolds. 
    
    This work partially extends the work of Ozuch in dimension 4. We are assuming the Einstein orbifold has negative scalar curvature, to facilitate the exposition. Also, this work includes an obstruction found by Morteza-Viaclovski as a special case, which holds when the blow-up limit comes from the Calabi ansatz.
\end{abstract}
\section{Introduction}
An Einstein manifold is a Riemannian manifold $(M,g)$ satisfying
\[\Ric(g)=\mu g\]
for some constant $\mu\in\R$. They arise naturally as critical points of the normalized Einstein-Hilbert functional, and appear as candidates for optimal metrics on manifolds in general dimensions $m>3$.
\subsection{Noncollapsing limits of compact Einstein manifolds}
We consider $m-$dimensional compact Einstein manifolds $(M,g)$ satisfying the following conditions:
\begin{equation}\label{moduli}
    |\Ric_g|_g\leq\Lambda,\quad\mathrm{diam}(M,g)\leq D,\quad\V(M,g)\geq V,\quad\int_M|\Rm|^{\frac{m}{2}}_gd\V_g\leq E
\end{equation}
for positive constants $\Lambda,D,V,E$, and denote $\M(m,\Lambda,D,V,E)$ to be the space of all $m-$dimensional compact Einstein manifolds satisfying (\ref{moduli}), modulo isometry. Endowing the moduli space with the Gromov-Hausdorff topology, there is a famous compactification result:
\begin{theorem}[\cite{Anderson,BKN}]\label{BKN}
    For any sequence $(M_i,g_i)_i$ in $\M(m,\Lambda,D,V,E)$, we can extract a subsequence converging to a compact metric space $(M_0,d)$ in Gromov-Hausdorff topology. Moreover,
    \begin{itemize}
        \item[(1)] $(M_0,d)$ is an $m$-dimensional Einstein orbifold with a finite singular set $S\subset M_0$;
        \item[(2)] The convergence is $C^{\infty}$ on $M_0\setminus S$;
        \item[(3)] For each $p\in S$, there are points $p_i\in M_i$ with $p_i\to p$, and positive numbers $t_i\to0$, such that the blow-up sequence $(M_i,t_i^{-2}g_i,p_i)$ converges in the pointed Gromov-Hausdorff sense. The blow-up limit $(N^m,g_b)$ satisfies
        \begin{itemize}
            \item[(i)] $(N,g_b)$ is asymptotically locally Euclidean (see Definition \ref{ALEdefn});
            \item[(ii)] $\Ric_{g_b}=0$;
            \item[(iii)] $\int_N|\Rm|_{g_b}^{\frac{m}{2}}d\V_{g_b}<+\infty$.
        \end{itemize}
    \end{itemize}
\end{theorem}
Note that in general, a blow-up limit $(N,g_b)$ may not be smooth. It would admit finitely many orbifold singularities, and we can do further blow-ups at singularities of $(N,g_b)$, until we get smooth limits. This is the bubble tree pattern of noncollapsing limits of compact Einstein manifolds.

We have added the finite energy assumption $\int_M|\Rm|^{\frac{m}{2}}_gd\V_g\leq E$ on the moduli space, which guarantees that the limit orbifold has only finitely many singular points. If such an assumption is removed, the singular set in the limit space would have Hausdorff dimension $m-4$ in general, as proved by \cite{CheegerNaber,JiangNaber}. Moreover, in dimension $m=4$, the finite energy assumption is superfluous: \cite{JiangNaber} showed that bounded Ricci curvature and volume non-collapsing imply $L^2$-finiteness of $|\Rm|$.  

Based on Theorem \ref{BKN}, we only consider Einstein orbifolds $M_0$ with finitely many singularities (which has dimension 0) throughout this paper. For brevity, we simply refer to them as Einstein orbifolds. More restrictively, we only treat special singular points $p$ on $M_0$, such that the blow-up sequence $(M_i,t_i^{-2}g_i,p_i)$ converges to a \textbf{smooth} Ricci-flat ALE manifold.

\subsection{Resolution of Einstein orbifolds}
In view of Theorem \ref{BKN}, we say an Einstein orbifold $(M_0,g_0)$ \textbf{admits a resolution sequence} if there is a sequence $\{(M_i,g_i)\}$ in $\M(m,\Lambda,D,V,E)$ converging to $(M_0,g_0)$. In other words, $(M_0,g_0)$ admits a resolution sequence if and only if it lies in the compactified moduli space $\overline{\M(m,\Lambda,D,V,E)}$. There is a long-standing question: what are the obstructions for a compact Einstein orbifold admitting a resolution sequence?

In a series of works \cite{Biq1,Ozu2,Ozu3}, Biquard and Ozuch give a partial answer to the question above, in dimension $m=4$. They found a local obstruction, which relates the curvature at the singularity $p\in M_0$ with geometric quantities of the blow-up limit at $p$. 

Recently, Wang and Yin's work \cite{WangYin} introduces the asymptotic curvature and the renormalized volume of a Ricci-flat ALE space in general dimensions. Such quantities enables us to formulate an analogous local obstruction in general dimensions. 

In this paper, we partially extend the Biquard-Ozuch's local obstruction to all dimensions $m\geq4$, closely following Ozuch's approach \cite{Ozu2,Ozu3}. For expositional coherence and to narrow the scope of our discussion, we additionally restrict our attention to Einstein orbifolds with negative scalar curvature, which enables us to work in Bianchi gauge (see the explanation following Proposition \ref{prop-gauge-linear}). Beyond the parallel extension to arbitrary dimensions and the unification of existing results, our study contains no further conceptual or technical innovations.

Before stating the main theorem, we explain some necessary notations. 

At an orbifold point $p\in(M_0,g_0)$, we denote $W(0)$ as the Weyl tensor at $p$, and $W_{ijkl}(0)$ as its components under geodesic coordinates centered at $p$. For a Ricci-flat ALE manifold $(N,g_b)$, \cite{WangYin} constructed optimal ALE coordinates on the end of $(N,g_b)$, that is, ALE coordinates satisfying Theorem \ref{ALEcoor}(1). With the help of such coordinates, they defined the asymptotic Weyl tensor $W^\infty$ and the renormalized volume $\mathcal{V}$ of $(N,g_b)$ (see Theorem \ref{ALEcoor} and Definition \ref{WVdef}). We will denote $W^{\infty}_{ijkl}$ as the components of $W^\infty$ under an optimal ALE coordinate chart.

Our main theorem is
\begin{theorem}\label{Maintheorem}
    Consider a compact Einstein orbifold $(M_0^m,g_0)$ satisfying the following conditions:
    \begin{itemize}
        \item $\Ric_{g_0}=\mu g_0$ with $\mu<0$;
        \item $(M_0,g_0)$ admits a resolution sequence $\{(M_i,g_i)\}\subset\M(m,\Lambda,D,V,E)$;
        \item at a singular point $p\in M_0$ modelled on $\R^m/\Gamma$, a blow-up sequence $(M_i,t_i^{-2}g_i,p_i)$ converges to a \emph{smooth} blow-up limit $(N,g_b)$.
    \end{itemize}
    Then, under any geodesic coordinates centered at $p$ and any optimal ALE coordinates satisfying Theorem \ref{ALEcoor}(1) for $(N,g_b)$, we have
    \begin{equation}\label{explicit-obstruction}
        \mu\mathcal{V}+\frac{\omega_{m-1}}{6m(m-2)|\Gamma|}(W_{ikjl}(0)W^\infty_{ikjl}+W_{ikjl}(0)W^\infty_{iljk})=0.
    \end{equation}
\end{theorem}
\begin{remark}
    Note that there is extra freedom for choices of geodesic coordinates at $p$ and optimal ALE coordinates of $(N,g_b)$. Different choices of coordinates cause the effects
    \[W_{ijkl}(0)\mapsto \tilde{W}_{abcd}(0)=W_{ijkl}(0)A_{ai}A_{bj}A_{ck}A_{dl},\]
    \[W^{\infty}_{ijkl}\mapsto \tilde{W}^{\infty}_{abcd}={W}^{\infty}_{ijkl}A'_{ai}A'_{bj}A'_{ck}A'_{dl},\]
    for some $A,A'\in\mathrm{O}(m)$. Our theorem claims that \eqref{explicit-obstruction} holds for coefficients in all such coordinates. In other words, there is a family of (not necessarily linearly independent) local obstructions, parametrized by $A^TA'\in\mathrm{O}(m)$.
\end{remark}

Now we discuss some applications of this theorem. It is known from \cite{BH,WangYin} that, if the blow-up limit $(N,g_b)$ is not flat, $\mathcal{V}$ is always strictly less than 0. As a corollary, since $\mu<0$, if $W(0)=0$, then \eqref{explicit-obstruction} is violated by all nontrivial blow-up limits $(N,g_b)$.  
\begin{theorem}\label{sphere}
    Suppose $(M_0^m,g_0)$ is a compact hyperbolic (constant sectional curvature $-1$) orbifold with a singular point modelled on $\R^m/\Z_2$. Then $(M_0,g_0)$ does not admit any resolution sequence.
\end{theorem}
The advantage of $\Z_2$ singularity is that, the blow-up limit at such a singularity must be smooth (as we will see in the proof of Theorem \ref{sphere}). For general singularities modelled on $\R^m/\Gamma$, we need to treat bubble trees. One can see \cite{HO} for discussions on resolving $S^4/\Gamma$ by bubble trees.

Another application is that, in the special case where the blow-up limit is the space given by the Calabi ansatz: $(N,g_b)=(O_{\CP^{n-1}}(-n),g_{Cal})$, \cite{MV} has computed the same type of our obstruction as \eqref{explicit-obstruction}.

Our exposition systematically unifies the key results and arguments scattered across \cite{Ozu2,Ozu3}, within a consistent notational framework adapted to general dimensions. All technical arguments in this work closely follow the original approach of Ozuch, and no new foundational tools or improved techniques are introduced. Our presentation aims to provide a self-contained, unified higher-dimensional treatment that can serve as a convenient reference for future related studies.

\subsection{Outline of the paper}

\begin{itemize}
    \item In Section 2, we provide preliminaries for Ricci-flat ALE manifolds, which arises as blow-up limits of compact Einstein manifolds. In particular, we will give the definition of asymptotic Weyl curvature and renormalized volume, which appeared in the statement of Theorem \ref{Maintheorem}.
    \item Based on the neck analysis conducted in \cite{Ozu1}, the resolution sequence $\{(M_i,g_i)\}$ is close to the gluing of $(M_0,g_0)$ and $(N,g_b)$ in some weighted $C^{2,\alpha}$-topology. (For the explicit statement, see Theorem \ref{neckana}.) Therefore, in the second part of Section 2, we set up the metrics and function spaces on the glued manifolds.
    \item Therefore, to prove Theorem \ref{Maintheorem}, we actually seek necessary conditions for the existence of an Einstein metric near the glued manifold $(M,g_t)$. To this end, we employ the deformation theory of Einstein manifolds. First, we establish a local slice theorem for the Bianchi gauge, which only applies to negative Einstein constants. This is done in Section 3.
    \item Next, we search for Einstein metrics near the glued metric $g_t$ that  lie in the Bianchi gauge of $g_t$. This is equivalent to deforming $g_t$ into a solution of
    \[\Ric(g)- \frac{1}{m }\left( \frac{1}{\V(g)} \int_M \s_g d\V_{g} \right) g + \delta_{\gtd}^* B_{g_t}(g)=0,\]
    which is a nonlinear elliptic equation. The infinitesimal deformation space is given by the kernel of its linearization, and the obstruction space is given by its cokernel. 
    \item Since linearized operator is self-adjoint, the infinitesimal deformation space and the obstruction space coincide. In Section 4, we provide a proper approximation of this space on the glued manifold, which is more explicit than the genuine obstruction space. Moreover, we analyze the linearized operator on the glued manifolds, a central step in deformation theory. 
    \item In Section 5, we solve the (obstructed) nonlinear equation above using the implicit function theorem and Lyapunov-Schmidt reduction. This step serves to describe the moduli space of genuine deformations, commonly referred to as the Kuranishi model in deformation theory. There is an obstructed Einstein metric corresponding to each infinitesimal deformation, and the genuine Einstein deformations correspond to the zeros of the Kuranishi map.
    \item Subsequently, we refine the gluing metric using (obstructed) Einstein deformations on the compact orbifold and the blow-up limit. Further more, since the Ricci curvature vanishes on the blow-up limit, which will cause a large error in the general case, we will add a curvature term to the blow-up limits by solving the linearized Ricci equation on the blow-up limit. We then demonstrate that the refined gluing metric provides a much better approximation of the resolution sequence. This is detailed in Section 6.
    \item Summarizing the discussions above, we formulate an explicit obstruction for admitting a resolution sequence, consisting of obstructions to Einstein deformations and obstructions to adding curvature on blow-up limits. By carefully analyzing these obstructions, we prove Theorems \ref{Maintheorem} and \ref{sphere}. This is done in Section 7. We also explain why the obstruction in \cite{MV} represents a special case of our obstruction.
\end{itemize}

\subsection{Index of notations} We list some notations appearing repeatedly in this paper.
\begin{itemize}
    \item $r_b$: Radius function on Ricci-flat ALE manifold $(N,g_b)$, defined in Notation \ref{r_b};
    \item $C^{k,\alpha}_\beta(g_b)$: Weighted H\"older space on Ricci-flat ALE manifold $(N,g_b)$, defined above Notation \ref{bHolder};
    \item $\rtd$: Radius function on $M_0\#N$, defined in \eqref{r_t};
    \item $\gtd$: Gluing metric of $(M_0,g_0)$ and $(N,g_b)$, with parameters $0<t=\frac{\delta^2}{16}<\delta^2\ll1$, defined in \eqref{naive-gluing};
    \item $C^{k,\alpha}_{\beta,\beta'}(\gtd)$: Weighted H\"older space on $(M_0\#N,\gtd)$, defined above Remark \ref{rmk3.1};
    \item $r_tC^{k,\alpha}_{\beta,*}$, $C^{k,\alpha}_{\beta,*}$: Weighted decoupling norms on vector fields and symmetric 2-tensors, defined in Definitions \ref{vector-norms}, \ref{tensor-norms};
    \item $P_{g}$: An elliptic linear operator related to infinitesimal Einstein deformation, defined in Notation \ref{P_g};
    \item $\SO(g_b)$: Decaying infinitesimal Einstein deformation in Bianchi gauge, defined in Definition \ref{bIED};
    \item $\SO(g_0)$: Volume preserving infinitesimal Einstein deformation in Bianchi gauge, defined in Notation \ref{IED0};
    \item $\TO(\gtd)$, $\TO(\gtd)^\perp$: Approximate infinitesimal Einstein deformations of $\gtd$, and its $L^2(\gtd)$ orthogonal complement, defined in Definition \ref{AIED};
    \item $\pi$, $\pi^\perp$: $L^2$-orthogonal projection onto $\TO(\gtd)$ and $\TO(\gtd)^\perp$, defined in Notation \ref{pi};
    \item $\Phi$, $\Psi$, $\Psi'$: Einstein operator with Bianchi gauge on $\gtd$, $g_b$, $g_0$, defined above Proposition \ref{Phi=0}, Proposition \ref{bdeform}, Proposition \ref{0deform};
    \item $h_{v}$, $h_{w}$, $g_{b,v}$, $g_{0,w}$: $h_{v}$ and $h_{w}$ are obstructed Einstein deformations of $g_b$, $g_0$, provided by Proposition \ref{bdeform}, Proposition \ref{0deform}. $g_{b,v}=g_b+h_v$, $g_{0,w}=g_0+h_w$.
    \item $H$, $h_2$: $H$ is the expansion of $g_0$ at orbifold point $p$, in a Bianchi coordinate, defined in \eqref{H}. $h_2$ is a solution of linearized Ricci equation, asymptotic to $H$, provided by Proposition \ref{h2};
    \item $\gtdp$: Gluing metric of $(M_0,g_{0,w})$ and $(N,g_{b,v}+t^2h_2)$, defined in \eqref{better-gluing}.
\end{itemize}

{\bf Acknowledgement}:
The author is grateful to the anonymous reviewer for careful reading and for valuable suggestions and comments, which have helped to correct several inaccuracies in the original arguments and to improve the presentation of this paper.

The author would like to thank his advisor, Professor Bing Wang, for guidance throughout this work, as well as for his constant support, encouragement, and helpful discussions during the research. The author is grateful to Professor Yu Li for reading the early version of this article and providing helpful suggestions.

\section{Preliminaries and Set-up}
\subsection{ALE manifolds}
\subsubsection{ALE coordinates and weighted norms}
\begin{definition}[ALE spaces]\label{ALEdefn}
    Given a complete Riemannian manifold/orbifold $(N^m,g)$ with finite singularities, we call it \textbf{asymptotically locally Euclidean (ALE)} of order $\tau>0$ if
    \begin{itemize}
        \item there exists a finite subgroup $\Gamma$ of $\mathrm{O}(m)$, $B(0,R)\subset\R^m/\Gamma$ and a compact subset $K\subset N$, such that there is a diffeomorphism 
        $$\varphi:(\R^m\setminus B(0,R))/\Gamma\to N\setminus K;$$
        \item under the coordinate $\varphi$, the metric is asymptotic to Euclidean metric $g_E$ in the following sense: for any $k\in\mathbb{N}\cup\{0\}$,
        \[|\nabla^k_{g_E}(\varphi^*g-g_E)|_{g_E}=O(|x|^{-\tau-k})\quad\text{as }|x|\to\infty.\]
    \end{itemize}
\end{definition}
\begin{notation}\label{r_b}
    For an ALE manifold \((N,g)\) with a given ALE coordinate chart \(\varphi: (\mathbb{R}^m \setminus B(0,R))/\Gamma \to N \setminus K\), we define a radius function \(r_b \in C^\infty(N)\) as follows. For \(x \in N \setminus K\), define \(r_b(x) := |\varphi^{-1}(x)|_{g_{E}}\); then extend it to be a smooth function on the whole manifold \(N\), such that \(1 < r_b(x) \leq R\) for all \(x \in K\).
\end{notation}

Now we define weighted Hölder norms on ALE manifolds. Given \(k \in \mathbb{N} \cup \{0\}\), \(\alpha\in(0,1],\, \beta \in \mathbb{R}_{>0}\), for a tensor field \(s\),
\[
[s]_{C^\alpha(g)}(x) := \sup_{\substack{y \in T_xM \\ |y| \leq \mathrm{inj}_g(x)}} \frac{|\exp_x^* s(0) - \exp_x^* s(y)|_{\exp_x^* g}}{|y|_{\exp_x^* g}^\alpha},
\]
\[
|s|_{C_\beta^k(g)}(x) := r_b(x)^\beta \sum_{i=0}^k r_b(x)^i |\nabla^i s(x)|_g,
\]
\[
\|s\|_{C_\beta^k(g)} := \sup_{x \in N} |s|_{C_\beta^k(g)}(x), \quad \|s\|_{C_\beta^{k,\alpha}(g)} := \sup_{x \in N} \left( |s|_{C_\beta^k(g)}(x) + r_b(x)^{k+\alpha+\beta} [\nabla^k s]_{C^\alpha(g)} \right).
\]
The space of \(C_\beta^k\) and \(C_\beta^{k,\alpha}\) tensor fields are defined as the completion of \(C_c^\infty\) tensor fields under the corresponding norms.

\begin{notation}\label{bHolder}
    Given ALE coordinate chart $\varphi:(\R^m\setminus B(0,R))/\Gamma\to N\setminus K$, for a function or a tensor field $s$ on $N$, we say
    $$s=O(r^{-\beta})\quad\text{if }s\in C^{k}_{\beta}, k=0,1,2,\dots$$
\end{notation}

\begin{lemma}[{\cite[Theorem 8.3.6]{Joyce}}]\label{Fredholm-on-blow-up limit}
    Let $k\geq2$ be an integer and $\alpha\in(0,1)$. For the Beltrami-Laplacian (acting on functions) $\Delta:C^{k,\alpha}_{\beta}\to C^{k-2,\alpha}_{\beta+2}$, we have:
    \begin{itemize}
        \item[(1)] When $\beta\in(0,m-2)$, $\Delta$ is an isomorphism with a bounded inverse. More explicitly, there exists $C>0$ such that for each $f\in C^{k-2,\alpha}_{\beta+2}$, there exists a unique $u\in C^{k,\alpha}_{\beta}$ such that $\Delta u=f$, and 
        \[\|u\|_{C^{k,\alpha}_{\beta}}\leq C\|f\|_{C^{k-2,\alpha}_{\beta+2}}.\]
        \item[(2)] When $\beta\in(m-2,m-1)$, there exist $C_1,C_2>0$ such that for each $f\in C^{k-2,\alpha}_{\beta+2}$, there exists a unique $u\in C^{k,\alpha}_{m-2}$ such that $\Delta u=f$. Moreover, 
        \[u=br^{2-m}+v,\]
        where
        \[b=\frac{1}{2-m}\mathrm{Area}_{g_E}\left(S^{m-1}/\Gamma\right)^{-1}\int_Nfd\V_g,\quad v\in C^{k,\alpha}_{\beta},\]
        satisfying
        \[|b|\leq C_1\|f\|_{C^0_{\beta+2}},\quad \|v\|_{C^{k,\alpha}_{\beta}}\leq C\|f\|_{C^{k-2,\alpha}_{\beta+2}}.\]
    \end{itemize}
\end{lemma}

\subsubsection{Intrinsic geometric quantities}
For any Ricci flat ALE manifold $(N^m,g)$, it was proved by \cite{WangYin} that, an ALE coordinate can be further refined. In the refined ALE coordinate $\varphi$, they give an explicit expansion of $\varphi^*g$, as stated below. Also see \cite{BH} for the case $m=4$.
\begin{theorem}[{\cite[Theorems 1.2, 1.5, 1.7]{WangYin}}]\label{ALEcoor}
    Consider an arbitrary Ricci flat ALE manifold $(N^m,g)$.
    \begin{itemize}
        \item[(1)] There exists an ALE coordinate chart $$\varphi:(\R^m\setminus B(0,R))/\Gamma\to N\setminus K$$ such that:
    \begin{itemize}
        \item it is Bianchi, that is, $B_{g_E}(\varphi^*g)=0$;
        \item the metric tensor has an expansion $$g_{ij}=\delta_{ij}+W_{ikjl}\frac{x^kx^l}{r^{m+2}}+O(r^{-m-1}),$$where $W_{ikjl}$ are components of a Weyl tensor.
    \end{itemize}
        \item[(2)] For any coordinate chart $\varphi$ given by (1), we denote $\Omega_r:=N\setminus\varphi(\{|x|>r\})$ and $B_r:=\{|x|<r\}\subset\R^m/\Gamma$. The limit
        $$\mathcal{V}:=\lim_{r\to\infty}\V_g(\Omega_r)-\V_{g_E}(B_r)$$
        exists.
        \item[(3)] Moreover, the Weyl tensor $W$ and $\mathcal{V}$ in (1)(2) are independent of the choice of ALE coordinate charts. To be explicit, if there exist another coordinate system $(y^a)$ on $N\setminus K$ and Weyl tensor $\tilde{W}$ such that
    \[g_{ab}=\delta_{ab}+\tilde{W}_{acbd}\frac{x^cx^d}{r^{m+2}}+O(r^{-m-1}),\]
    then there exists $A\in\mathrm{O}(m)$ such that $W_{ijkl}=\tilde{W}_{abcd}A_{ai}A_{bj}A_{ck}A_{dl}$, and the two coordinate systems $(x^i)$ and $(y^a)$ give the same limit $\mathcal{V}$.
        \item[(4)] Finally, $\mathcal{V}\leq0$, and $\mathcal{V}=0$ if and only if $(N,g)$ is isometric to $(\R^m/\Gamma,g_E)$.
    \end{itemize}
\end{theorem}
\begin{definition}\label{WVdef}
    We call the Weyl tensor given by the theorem above as the \textbf{asymptotic Weyl curvature} of $(N,g)$, and we denote it by $W^\infty$. The limit $\mathcal{V}$ in the above theorem is called the \textbf{renormalized volume} of $(N,g)$.
\end{definition}
A wonderful observation of \cite{BH} is that, the renormalized volume appears in the asymptotic expansion of Poisson solution. Now we make it clear in general dimensions.
\begin{proposition}\label{Poisson-solution}
    For any Ricci flat ALE manifold $(N^m,g)$, fix an ALE coordinate chart given by Theorem \ref{ALEcoor}. There exists a unique solution to
\[
\begin{cases}
\Delta u = 2m, \\
u = r^2 + o(1).
\end{cases}
\]
Moreover, $u$ has the following asymptotic expansion on $N\setminus K$:
\[
u = r^2 + \frac{2m}{2 - m} \, \mathrm{Area}_{g_E}\left(S^{m-1}/\Gamma\right)^{-1}\mathcal{V}\cdot r^{2 - m} + O(r^{1 - m+\varepsilon}).
\]
\end{proposition}
\begin{proof}
    The uniqueness of $u$ is guaranteed by maximum principle, so it remains to prove the existence and asymptotic expansion.

    With help of the expansion in Theorem \ref{ALEcoor}(2),
\[
\Delta_gr^2=g^{ij}\left(\partial_i\partial_j-\Gamma_{ij}^k\partial_k\right)r^2=2m+O(r^{-m-1}).\]
So we only need to solve $v := u - r^2 = o(1)$, which satisfies 
\begin{equation}\label{equ2.2}
    \Delta_gv =2m-\Delta_gr^2= O(r^{-m-1}).
\end{equation}
Fix any $\varepsilon > 0$, lemma \ref{Fredholm-on-blow-up limit} tells that there exists $v =b r^{2 - m} + O(r^{1 - m + \varepsilon})$ satisfying \eqref{equ2.2}. Moreover, exhausting $N$ by $\Omega_r = N\setminus \varphi(\{|x|>r\})$, we have
\begin{align*}
    (2-m) \mathrm{Area}_{g_E}\left(S^{m-1}/\Gamma\right)\cdot b=&\int_N\Delta_gv\,d\V_g\\
    =&\lim_{r\to\infty}\int_{\Omega_r}(2m-\Delta_gr^2)d\V_g\\
    =&\lim_{r\to\infty}2m\V_g(\Omega_r)-\int_{\partial \Omega_r}\langle \nabla r^2, n \rangle (n\lrcorner d\V_g)\\
    =&\lim_{r\to\infty}2m\V_g(\Omega_r)-\int_{\partial B_r} \left(2r + O(r^{-m})\right)(\partial_r\lrcorner d\V_{g_E})\\
    =&\lim_{r\to\infty}2m\V_g(\Omega_r)-2m\V_{g_E}(B_r)+O(r^{-1})\\
    =&2m\mathcal{V}.
\end{align*}
Hence
\[
u=r^2+v=r^2+b r^{2 - m} + O(r^{1 - m + \varepsilon})
\]
with $b = \frac{2m}{2 - m}\mathrm{Area}_{g_E}\left(S^{m-1}/\Gamma\right)^{-1}\mathcal{V}$.
\end{proof}

\subsection{Gluing construction}
Recall that our assumption is, there is a sequence $\{(M_i,g_i)\}$ in $\M(m,\Lambda,D,V,E)$ converging to an Einstein orbifold $(M_0,g_0)$, and there is a singular point $p\in M_0$ which admits a smooth blow-up limit $(N,g_b)$. In this subsection, we consider a converse process, that is, glue $(N,g_b)$ to the singular point $p\in M_0$. For simplicity of notations, we assume that $p$ is the only singular point of $M_0$, modeled on $\R^m/\Gamma$.

As a remark, we can also consider the general case, that $(M_0,g_0)$ has finitely many orbifold points, and there is a bubble tree formation at each orbifold point. Now we glue blow-up limits on different layers inductively, and the resulting manifold is still smooth and compact. So one can see that all discussions in this subsection also extend to the general case.

\subsubsection{Gluing different regions}
For parameters \( 0 < t<\delta^2\ll 1 \), in particular, \( \delta <\mathrm{inj}_{g_0}(p) \), we describe the gluing of $(M_0,g_0)$ and $(N,g_b)$, and such a construction is now well established in the literature. For simplicity and clarity, we set $\delta^2=16t$ throughout this paper.

We have geodesic coordinates around the singular point \( p \in M_0 \),
\[
\psi: B(0,\delta)/\Gamma \to \{ d(p,\cdot) < \delta \} \subseteq M_0,
\]
and ALE coordinates of \( N \), 
\[ \varphi: (\R^m\setminus B(0,\delta^{-1}))/\Gamma \to \{ r_b \geq \delta^{-1} \} \subseteq N. \]
Via the dilation \( S_t: x \mapsto tx \) on \( \mathbb{R}^m/\Gamma \), we have a diffeomorphism:
\[
\psi \circ S_t \circ \varphi^{-1}: \{ x\in N| \delta^{-1} \leq r_b(x) < t^{-1} \delta \}  \to \{ y\in M_0|t \delta^{-1} \leq d(p,y) < \delta \}.
\]
Then we can patch \( \{  r_b < t^{-1}\delta \} \subseteq N \) and \( \{ d(p,\cdot)\geq t \delta^{-1}  \} \subseteq M_0 \) together, by identifying the two regions above. More explicitly, the glued manifold is diffeomorphic to \( M_0 \# N \), and consists of three disjoint regions:
\[
\begin{aligned}
&N^\delta \cong \{ x\in N| r_b(x) < \delta^{-1} \} , \\
&A(t,\delta) \cong \{ x\in N|\delta^{-1} \leq r_b(x) < t^{-1} \delta\}\cong\{y\in M_0|t \delta^{-1} \leq d(p,y) < \delta \}, \\
&M_0^\delta \cong \{y\in M_0 | d(p,y) \geq \delta \}.
\end{aligned}
\]

\begin{figure}[ht]
    \centering
    \includegraphics[scale=0.8]{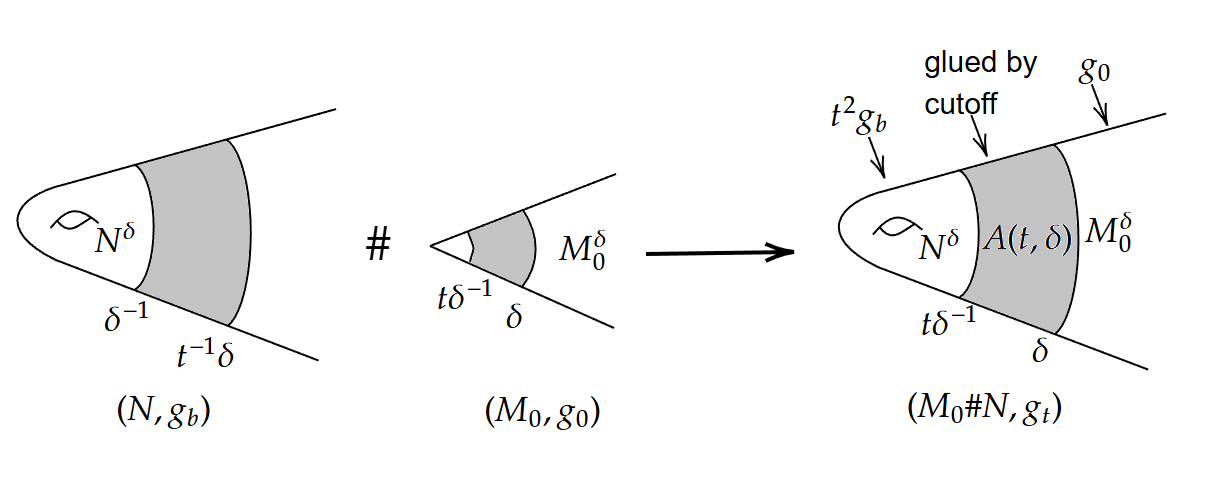}
    \caption{The gluing construction}
\end{figure}

Next, we define a radius function on the glued manifold \( M_0 \# N \):
\begin{equation}\label{r_t}
    r_{t} :=
    \begin{cases}
    t r_b, & \text{on } N^\delta \\
    t r_b = d(p,\cdot), & \text{on } A(t,\delta) \\
    d(p,\cdot), & \text{on } M_0^\delta
    \end{cases}.
\end{equation}

Choose a cutoff function \( \chi: \mathbb{R}_{\ge 0} \to \mathbb{R}_{\ge0} \), satisfying \( \chi \equiv 1 \) on \( [0, 1] \) and \( \chi \equiv 0 \) on \( [2, \infty) \). Then we define $\chi_t:M_0\#N\to\R_{\ge0}$ by $\chi_t=\chi\left(\frac{r_t}{\sqrt{t}}\right)$, so $\chi_t\equiv 1$ on $N^\delta$ and $\chi_t\equiv0$ on $M_0^\delta$. Using such a cutoff, we define the gluing metric on \( M_0 \# N \):
\begin{equation}\label{naive-gluing}
    g_{t} := t^2 \chi_{t} g_b + \left( 1 - \chi_{t} \right) g_0.
\end{equation}
Now we have the glued Riemannian manifold \( (M_0 \# N, g_{t}) \). 

\begin{notation}\label{cut-on-neck}
For clarity and for later use, we describe different regions and cutoff functions in terms of $r_t$.
\begin{itemize}
    \item As defined before, 
    \[N^\delta=\{r_t<t\delta^{-1}\}=\left\{r_t<\frac{\sqrt{t}}{4}\right\},\]
    \[A(t,\delta)=\{t\delta^{-1}\le r_t<\delta\}=\left\{\frac{\sqrt{t}}{4}\le r_t<4\sqrt{t}\right\},\]
    \[M_0^\delta=\{r_t\ge\delta\}=\{r_t\ge4\sqrt{t}\}.\]
    \item As defined before, $\chi_t=\chi\left(\frac{r_t}{\sqrt{t}}\right)$ equals to $1$ on $\{r_t\le\sqrt{t}\}\supsetneqq N^\delta$, and is supported on $\{r_t<2\sqrt{t}\}\subsetneqq N^\delta\cup A(t,\delta)$.
    \item We define 
    \[\chi_{A(t,\delta)}:=\chi\left(\frac{4r_t}{\sqrt{t}}\right)-\chi\left(\frac{r_t}{2\sqrt{t}}\right).\]
    It is supported on $A(t,\delta)$ and equals to $1$ on $\{\frac{\sqrt{t}}{2}\leq r_t\leq 2\sqrt{t}\}$.
\end{itemize}
\end{notation}

\subsubsection{Weighted H\"older norms on glued manifold}
The neck region connects singular point of \( (M_0, g_0) \) and infinity of \( (N, g_b) \), so we need a weighted Hölder norm that detects both the growth near orbifold point and the decay on ALE end.

As preparation, we recall the definition of weighted Hölder norms on orbifold \( (M_0, g_0) \). Given \( k \in \mathbb{N} \cup \{0\} \), \( \alpha \in (0,1] \), \( \beta \in \mathbb{R}_{>0} \), for a tensor field \( s \):
\[
[s]_{C^\alpha} (x) = \sup_{\substack{y \in T_xM_0 \\ |y|\leq\mathrm{inj}_{g_0}x}} \frac{ \left| \exp_x^* s(x) - \exp_x^* s(y) \right|_{\exp_x^* g_0} }{ |y|^\alpha_{\exp_x^* g_0} },
\]
\[\| s \|_{C^{k}_{\beta}(g_0)} := \sup_{x \in M_0} d(x,p)^{-\beta}  \sum_{i=0}^k d(x,p)^i |\nabla^i s(x)| ,\]
\[
\| s \|_{C^{k,\alpha}_{\beta}(g_0)} := \sup_{x \in M_0} d(x,p)^{-\beta} \left( \sum_{i=0}^k d(x,p)^i |\nabla^i s(x)| + d(x,p)^{k+\alpha} [\nabla^ks]_{C^\alpha} (x) \right).
\]

For a \( (p,q) \)-tensor field \( s \) on \( (M_0 \# N, g_{t}) \), we decompose
\[
s = \chi_{t} s + (1 - \chi_{t}) s.
\]
Since \( \text{supp}(\chi_{t}) \subset N^\delta \cup A(t,\delta) \), \( \text{supp}(1 - \chi_{t}) \subset A(t,\delta) \cup M_0^\delta \), we can view \( \chi_{t} s \) as a tensor field on \( (N, t^2 g_b) \), and \( (1 - \chi_{t}) s \) as a tensor field on \( (M_0, g_0) \). Thus for \( k \in \mathbb{N} \cup \{0\} \), \( \alpha \in (0,1] \), \( \beta_1, \beta_2 \in \mathbb{R}_{>0} \), we define the weighted Hölder norm of \( (p,q) \)-tensor field \( s \) as
\[
\begin{aligned}
\| s \|_{C^{k,\alpha}_{\beta_1,\beta_2}(g_{t})} &:= \| \chi_{t} s \|_{C^{k,\alpha}_{\beta_1}(t^2 g_b)} + \| (1 - \chi_{t}) s \|_{C^{k,\alpha}_{\beta_2}(g_0)} \\
&= t^{p - q} \| \chi_{t} s \|_{C^{k,\alpha}_{\beta_1}(g_b)} + \| (1 - \chi_{t}) s \|_{C^{k,\alpha}_{\beta_2}(g_0)},
\end{aligned}
\]
\[\|s \|_{C^{k}_{\beta_1,\beta_2}(g_{t})} := \| \chi_{t} s \|_{C^{k}_{\beta_1}(t^2 g_b)} + \| (1 - \chi_{t}) s \|_{C^{k}_{\beta_2}(g_0)}.\]

\begin{remark}\label{rmk3.1}
    (i) Since \( M_0 \) is compact, \( \| s \|_{C^0_{\beta}(g_0)} \leq 1 \) is equivalent to: 
    \[|s| \leq r^{\beta}\text{ near }p,\text{ and }|s|\text{ is bounded on }M_0.\]
    (ii) For a tensor field \( s \) on \( (M_0 \# N, g_{t}) \), \( \| s \|_{C^{0}_{\beta_1,\beta_2}(g_{t})} \leq 1 \) is equivalent to \( |s| \leq f\circ r_{t} \), where
    \[
    f(r) =
    \begin{cases}
    r^{\beta_2} + \left(\frac{t}{r}\right)^{\beta_1}, & \text{for } t \delta^{-1} \leq r \leq \delta \\
    (t \delta^{-1})^{\beta_2} + \delta^{\beta_1}, & \text{for } r < t \delta^{-1} \\
    \delta^{\beta_2} + (t \delta^{-1})^{\beta_1}, & \text{for } r > \delta
    \end{cases}.
    \]
\end{remark}

\begin{remark}
    (1) We sometimes use the notation $s\in\rtd^{-l}C^{k,\alpha}_{\beta_1,\beta_2}(g_{t})$ for $l\in\R$, which means that $\rtd^ls\in C^{k,\alpha}_{\beta_1,\beta_2}(g_{t})$, and 
    \[\|s\|_{\rtd^{-l}C^{k,\alpha}_{\beta_1,\beta_2}(g_{t})}:=\|\rtd^ls\|_{C^{k,\alpha}_{\beta_1,\beta_2}(g_{t})}.\]

    (2) From now on, we consider weighted H\"older norms $C^{k,\alpha}_{\beta,\beta}$ on $g_t$, that is, $\beta_1=\beta_2$ (denoted by $\beta$), with fixed $0<\beta<1$, unless otherwise specified.
\end{remark}

\subsubsection{Approximating a resolution sequence}
Now let's go back to our question, about resolution of Einstein orbifolds. Recall that we suppose \( (M_0, g_0) \) is a compact Einstein orbifold with only one singular point \( p \), and it admits a resolution sequence \( \{ (M_i, g_i) \}\subseteq \mathcal{M}(m, \Lambda, D, V, E) \) with a smooth blow-up limit \( (N, g_b) \).

First of all, after choosing a subsequence, we can assume each $M_i$ is diffeomorphic to $M_0\#N$. On the neck region, Bando \cite{Bando90,BandoCorrection} proved that Riemannian curvature of $g_i$ is decaying as $r^{\beta}+\left(\frac{t}{r}\right)^\beta$. Then, the following fundamental theorem of Ozuch \cite{Ozu1} tells that, the (pointed-)Gromov-Hausdorff approximation could be improved to weighted $C^{2,\alpha}$ approximation. 

\begin{theorem}[\cite{Ozu1}, Theorem 6.4; \cite{Ozu2}, Corollary 2.10]\label{neckana}
    There exists a constant $\beta>0$ such that the followings hold.

    For any \(  \varepsilon > 0 \), \( \alpha \in (0,1] \), there exists \( i_0 > 0 \), such that for all \( i \geq i_0 \), there are parameters \( 0 < t_i=\frac{\delta_i^2}{16} \ll 1 \) and diffeomorphisms \( \phi_i: M_0 \# N \to M_i \), with
    \[
    \| \phi_i^* g_i - g_{t_i} \|_{C^{2,\alpha}_{\beta,\beta}(g_{t_i})} <\varepsilon.
    \]
    Moreover, $t_i,\delta_i$ tend to $0$ as $i\to\infty$.
\end{theorem}

\section{Gauge fixing}
The main aim of this section is to show that, for $t$ sufficiently small, up to diffeomorphism, each metric near $g_t$ lies in Bianchi gauge of $g_t$. 

First of all, for any Riemannian metric $g$, we define two operators acting on symmetric $(0,2)-$tensors as follows:
\[\delta_g:\Gamma(Sym^2(T^*M))\to\Omega^1(M),\quad (\delta_gh)_i:=-g^{jk}\nabla_kh_{ij},\]
\[B_g:\Gamma(Sym^2(T^*M))\to\Omega^1(M),\quad B_g h := \delta_g h + \frac{1}{2}d(\text{tr}_g h).\]
\begin{definition}\label{BG}
    We say a metric $g$ on $M$ is in \textbf{Bianchi gauge} of $g_t$ if $B_{g_t}(g)=0$.
\end{definition}

Also we introduce
\[\delta^*_g:\Omega^1(M)\to \Gamma(Sym^2(T^*M)),\quad (\delta^*_g\omega)_{ij}:=\frac{1}{2}(\nabla_i\omega_j+\nabla_j\omega_i),\]
which is the formal adjoint of $\delta_g$. Since $\Omega^1(M)$ is the dual space of $\Gamma(TM)$, we can also regard $\delta_g^*$ as an operator acting on vector fields, which differs from the Lie derivative of $g$ only by a factor $\frac{1}{2}$.

\subsection{Weighted decoupling norms}
We seek for a linear estimate, which serves for local slice theorem for Bianchi gauge. 

For each $X\in r_tC^{3,\alpha}_{\beta,\beta}(g_t)$, we denote $\exp_X:=\exp_{g_t}(X)$ to be the diffeomorphism defined by the exponential map of $g_t$. To apply the inverse function theorem, we need to estimate the linearized operator of $X\mapsto B_{g_t}\exp_X^*g$. In other words, we wish that for all sufficiently small $t$,
\[B_{g_t}\delta_{g_t}^*: r_tC^{3,\alpha}_{\beta,\beta}(g_t)\to r_t^{-1}C^{1,\alpha}_{\beta,\beta}(g_t)\]
is invertible. 

However, this is obstructed by non-trivial kernel on the neck region, which is modeled on $(\R^m-\{0\})/\Gamma$. More precisely, such obstruction consists of vector fields $X$ on $(\R^m-\{0\})/\Gamma$ such that
\[B_{g_E}\delta_{g_E}^*(X)=\frac{1}{2}\nabla^*_{g_E}\nabla_{g_E}(X)=0,\quad\text{and}\quad|X(x)|\leq C|x|\left(|x|^\beta+|x|^{-\beta}\right).\]
The good news is that, such vector fields only consist of linear vector fields on $(\R^m-\{0\})/\Gamma$. This leads to the weighted decoupling norms defined by Ozuch, which measures the non-trivial kernel above separately in unweighted norms. 
\begin{definition}[Weighted decoupling norms on vector fields]\label{vector-norms}
    Let $X$ be a vector field on $M=M_0\#N$, and recall we have defined a cutoff function $\chi_{A(t,\delta)}$ supported on the neck region $A(t,\delta)$ (see Notation \ref{cut-on-neck}). The $r_tC^{k,\alpha}_{\beta,*}-$norm of $X$ is defined as
    \[\|X\|_{r_tC^{k,\alpha}_{\beta,*}}=\inf\left\{\|X_*\|_{r_tC^{k,\alpha}_{\beta,\beta}}+\|X_A\|_{rC^0_0(g_E)}\left|\begin{aligned}
        X=X_*+\chi_{A(t,\delta)}X_A, \text{ where }X_A\\
        \text{ is a linear vector field on }\R^m/\Gamma.
    \end{aligned}\right.\right\}\]
\end{definition}

For similar reasons, Ozuch introduced the following decoupling norms on symmetric 2-tensors, which decouples Euclidean harmonic 2-tensors bounded by $|x|^{\beta}+|x|^{-\beta}$. 

\begin{definition}[Weighted decoupling norms on symmetric 2-tensors]\label{tensor-norms}
    Let $h$ be a symmetric 2-tensor field on $M=M_0\# N$. The $C^{k,\alpha}_{\beta,*}$-norm of $h$ is defined as
    \[\|h\|_{C^{k,\alpha}_{\beta,*}}=\inf\left\{\|h_*\|_{C^{k,\alpha}_{\beta,\beta}}+|h_A|_{g_E}\left|\begin{aligned}
        h=h_*+\chi_{A(t,\delta)}h_A, \text{ where }h_A\text{ is a }\\
        \text{ constant symmetric 2-tensor on }\R^m/\Gamma.
    \end{aligned}\right.\right\}\]
\end{definition}

\subsection{Local slice theorem}
The following proposition is essentially equivalent to {\cite[Lemma 3.7]{Ozu2}}. The proof can be obtained by carrying over Ozuch’s argument verbatim to the Bianchi gauge setting. For this reason, we only briefly explain why the original framework remains valid under our setup, and omit the detailed proof, which can be found in Ozuch’s original paper.

\begin{proposition}\label{prop-gauge-linear}
    There exists constants $0<\delta_0\ll1$ such that, for all $0<\delta<\delta_0$ and $t=\frac{\delta^2}{16}$, the linear operator (acting on vector fields)
    \[B_{g_t}\delta_{g_t}^*: r_tC^{3,\alpha}_{\beta,*}(g_t)\to r_t^{-1}C^{1,\alpha}_{\beta,*}(g_t)\]
    is invertible. Moreover, there exists $C>0$ such that
    \begin{equation}\label{equ-gauge-linear}
        \|X\|_{r_tC^{3,\alpha}_{\beta,*}(g_t)}\leq C\|B_{g_t}\delta_{g_t}^*X\|_{r_t^{-1}C^{1,\alpha}_{\beta}(g_t)},\quad \forall X\in r_tC^{3,\alpha}_{\beta,*}(g_t).
    \end{equation}
\end{proposition}
To prove this, we reduce it to the invertibility on $M_0$ and $N$. On the compact orbifold $M_0$, by Ricci identity,
\[B_{g_0}\delta^{*}_{g_0}=\frac{1}{2}\nabla^*_{g_0}\nabla_{g_0}-\frac{1}{2}\mu,\]
where $\mu$ is the Einstein constant of $g_0$. By our assumption, $\mu<0$, hence $B_{g_0}\delta^{*}_{g_0}$ is invertible. On the ALE manifold $N$, we have
\[B_{g_b}\delta^{*}_{g_b}=\frac{1}{2}\nabla^*_{g_b}\nabla_{g_b},\]
which has trivial kernel on $r_bC^{3,\alpha}_\beta(g_b)$, hence it is also invertible. Consequently, Ozuch’s proof carries over directly to our situation.

As a tool to solve nonlinear PDEs,  we give a quantitative version of implicit function theorem.
\begin{lemma}\label{IFT}
    Let $F: X \to Y$ be a smooth map between Banach spaces, and let $Q := F - F(0) - dF_0$. Suppose there are constants $q, r, c > 0$ such that:
    \begin{enumerate}
        \item[(1)] $\|Q(x) - Q(y)\| \leq q\|x - y\|(\|x\| + \|y\|)$, $\forall x, y \in B(0, r) \subset X$;
        \item[(2)] $dF_0$ is an isomorphism, and $\|(dF_0)^{-1}\| \leq c$;
        \item[(3)] $r < \min\left\{1, \frac{1}{2qc}\right\}$ and $\|F(0)\| \leq \frac{r}{2c}$.
    \end{enumerate}
    Then the equation $F(x) = 0$ has a unique solution in $B(0, r)$.
\end{lemma}
    
\begin{proof}
    We need to solve $F(x) = 0$, that is,
    \[
    F(0) + dF_0(x) + Q(x) = 0.
    \]
    
    First, let's solve the linear equation $dF_0(x) = y$. The assumption (2) tells there is a solution $(dF_0)^{-1}(y)$ for any $y \in Y$.
    
    Then we look for a solution $x$ for $dF_0(x) = y(x)$, where $y(x) = -F(0) - Q(x)$. This is equivalent to the equation 
    \begin{equation}\label{4.1}
        x = (dF_0)^{-1}(y(x)) = - (dF_0)^{-1}(F(0) + Q(x))
    \end{equation}
    
    Define $G: X \to X$, $G(x) = - (dF_0)^{-1}(F(0) + Q(x))$. One could verify that $G(B(0, r)) \subset B(0, r)$, and 
    \[
    \|G(x) - G(y)\| \leq 2rcq\|x - y\| < \|x - y\|,
    \]
    hence by Banach's fixed-point theorem, there exists a unique solution $x\in B(0,r)$ of (\ref{4.1}).
\end{proof}

By applying Lemma \ref{IFT} to $F:X\mapsto B_{g_t}\exp_X^*g$, with the help of Proposition \ref{prop-gauge-linear}, we immediately get the following slice theorem.
\begin{proposition}\label{Slice-theorem}
    There exists $\varepsilon>0$, $C>0$ and $0<\delta_0\ll1$ such that the followings hold.

    for any $0<\delta<\delta_0$, $t=\frac{\delta^2}{16}$, and any metric $g$ on $M$ with $\|g-g_t\|_{\ctab(g_t)}<\varepsilon$, there exists a diffeomorphism $\varphi:M\to M$ such that
    \[B_{g_t}(\varphi^*g)=0\quad\text{and}\quad\|\varphi^*g-g\|_{\ctab(g_t)}\leq C\|g-g_t\|_{\ctab(g_t)}.\]
    As a corollary, we have $\|\varphi^*g-g_t\|_{\ctab(g_t)}\leq (C+1)\|g-g_t\|_{\ctab(g_t)}$.
\end{proposition}

\section{Infinitesimal Einstein deformations}
In this section we study infinitesimal Einstein deformations (IED), which are solutions to the linearization of Einstein equation in Bianchi gauge. We will study IED on Ricci-flat ALE manifold $(N,g_b)$ and those on compact orbifold $(M_0,g_0)$. Then we approximate IED on the glued manifold $(M,g_t)$, again in a gluing manner.

\subsection{On Ricci-flat ALE}
Given a Ricci-flat ALE manifold $(N^m,g_b)$, suppose there is a family of Ricci-flat ALE metric $g_{s}$ with $g_{0}=g_b$. In this case, the infinitesimal deformation $h=\frac{d}{ds}\big|_{s=0}g_{s}$ must satisfies
$$d\Ric_{g_b}(h)=0\quad\text{and}\quad h=O(r^{-\varepsilon})\;\text{for some }\varepsilon>0.$$
On Ricci-flat manifolds,
\[
d\Ric_{g_b}(h)= \frac{1}{2}\nabla^*_{g_b} \nabla_{g_b} h - \mathring{R}_{g_b}h - \delta^*_{g_b}(B_{g_b} h) ,
\]
where $(\mathring{R}_{g_b}h)_{ij} := R_{ikjl}h^{kl}$, $(\delta_{g_b}h)_i:=-g_b^{jk}\nabla_kh_{ij}$, $B_{g_b} h := \delta_{g_b} h + \frac{1}{2}d(\text{tr}_{g_b} h)$ and $(\delta^*_{g_b}\omega)_{ij}:=\frac{1}{2}(\nabla_i\omega_j+\nabla_j\omega_i)$. 

By diffeomorphism invariance of $\Ric$, we know that for any diffeomorphism $\varphi:M\to M$, $\Ric(\varphi^*g_b)=\varphi^*\Ric(g_b)$. By slice theorem in weighted function space, we can reduce the kernel of $d\Ric_{g_b}$ to Bianchi gauge, that is, infinitesimal deformations transverse to $\text{Diff}(M)$-orbits, such that $B_{g_b} h = 0$. 

\begin{notation}\label{P_g}
    For an arbitrary Riemannian metric $g$, we denote 
    $$P_{g}:=\frac{1}{2}\nabla^*_{g} \nabla_{g} - \mathring{R}_{g}$$
    acting on symmetric $(0,2)-$tensor fields. Compared to $d\Ric_{g}$, it has an advantage that it is an elliptic operator.
\end{notation}

\begin{definition}\label{bIED}
    Given a Ricci-flat ALE manifold $(N,g_b)$, we define the space of \textbf{infinitesimal Einstein deformations} as
    \begin{align}
        \SO(g_b):=\{h\in\ker(P_{g_b})|h=O(r^{-\varepsilon})\text{ for some }\varepsilon>0\}.
    \end{align}
\end{definition}
It is known that $\SO(g_b)$ is indeed a subspace of $\ker(d\Ric_{g_b})$, and elements of $\SO(g_b)$ have more properties and better decay, as listed in the following proposition.
\begin{proposition}\label{IEDb}
    For Ricci-flat ALE manifold $(N^m.g_b)$, any $h\in\SO(g_b)$ satisfies
    \begin{enumerate}
        \item[(1)] $\delta_{g_b}h=0,\;\mathrm{tr}_{g_b}h=0$;
        \item[(2)] $B_{g_b}h=0,\;d\Ric_{g_b}h=0$;
        \item[(3)] $h=H^m+O(r^{-m-\varepsilon})$, where $H^m$ is a symmetric $(0,2)$ tensor with coefficients homogenous of order $r^{-m}$. In particular, $h$ is $L^2-$integrable on $(N,g_b)$.
    \end{enumerate}
\end{proposition}
\begin{proof}
    (1) On a Ricci-flat manifold, using second Bianchi identity and Ricci identities, we have
    $$\delta_{g_b}P_{g_b}=\frac{1}{2}\nabla^*_{g_b}\nabla_{g_b}\delta_{g_b},\quad \mathrm{tr}_{g_b}P_{g_b}=\frac{1}{2}\nabla^*_{g_b}\nabla_{g_b}\mathrm{tr}_{g_b}.$$
    Thus for $h\in\SO(g_b)$, both $\delta_{g_b}h$ and $\mathrm{tr}_{g_b}h$ are harmonic and decaying to $0$ at infinity. By the maximum principle, they vanish identically.

    (2) By definition,
    $$B_{g_b}h=\delta_{g_b}h+\frac{1}{2}d(\mathrm{tr}_{g_b}h)=0\quad\text{(using (1))},$$
    and 
    $$d\Ric_{g_b}h=P_{g_b}h-\delta^*_{g_b}(B_{g_b}h)=0.$$

    (3) Recall that $(g_b)_{ij}=\delta_{ij}+O(r^{-m})$ and $h_{ij}=O(r^{-\varepsilon})$. Hence  
    \[
    \Delta_{g_E} h_{ij} = (\delta^{kl} - g_b^{kl}) \nabla_k \nabla_l h_{ij} + g_b^{kl} \nabla_k \nabla_l h_{ij}=(\delta^{kl} - g_b^{kl}) \nabla_k \nabla_l h_{ij} + R_{ikjl}h^{kl} = O(r^{-m-2-\varepsilon}).
    \]  

    For any \( i,j \), there exists a smooth function \( k_{ij} \) on \( (\mathbb{R}^m \setminus B_1(0))/\Gamma \) such that  
    \[
    \begin{cases}
        \Delta_{g_E} k_{ij} = \Delta_{g_E} h_{ij}\\
        k_{ij} = O(r^{-m  - \varepsilon}).
    \end{cases}.
    \]

    Write \( h_{ij} - k_{ij} =: H_{ij} \). Each \( H_{ij} \) is a harmonic function on \( (\mathbb{R}^m \setminus B_1(0))/\Gamma \), hence it can be expanded in terms of harmonic polynomials of degree \( 2-m, 1 - m, -m, \cdots \).  Now 
    \[ h_{ij} = H_{ij} + O(r^{-m - \varepsilon}), \]
    and by (1) we have \( \delta_{g_b} h = 0 \). We claim that this rules out \( r^{2 - m} \) and \( r^{1-m} \) terms in \( H_{ij} \), which leads to the desired result.  

    Indeed, write  
    \[
    H_{ij} = a_{ij} r^{2 - m} + b_{ijk} \frac{x^k}{r^m} + O(r^{-m}), \quad a_{ij},b_{ijk}\in\R,
    \]  
    the leading terms of \( \delta_{g_b}(h) \) is  
    \[
    \begin{aligned}
        &\delta_{g_E}\left[ \left(a_{ij} r^{2 - m} + b_{ijk} \frac{x^k}{r^m}\right) dx^i\otimes dx^j \right] \\
        =&\left[ (2-m)a_{ij}x^jr^{-m}+b_{ijk}\left(\delta_{jk}r^{-m}-mx^jx^kr^{-m-2}\right)     \right]dx^i.
    \end{aligned}
    \]
    It vanishes identically, so $a_{ij}=0$ for each $i,j$, and 
    \[\sum_lb_{ill}r^2=m\sum_{j,k}b_{ijk}x^jx^k.\]
    This implies that for $j\neq k$, $b_{ijk}=0$; and for each $j$, $b_{ijj}=\frac{1}{m}\sum_lb_{ill}:=\tilde{b}_i$. Now
    \[H_{ij}=\tilde{b}_i\frac{x^j}{r^m}+O(r^{-m}).\]
    Since for each $i,j$, $0=H_{ij}-H_{ji}=\left(\tilde{b}_ix^j-\tilde{b}_jx^i\right)r^{-m}+O(r^{-m})$, we get each $\tilde{b}_i=0$, hence $H_{ij}=O(r^{-m})$.
\end{proof}

It turns out that, the Poisson solution $u$ in Proposition \ref{Poisson-solution} gives an explicit element of $\SO(g_b)$. 
\begin{proposition}
    For any Ricci flat ALE space $(N^m.g_b)$, we have
    \[(\mathscr{L}_{\nabla u}g_b)^{\circ}\in\SO(g_b),\]
    where $(\cdot)^{\circ}$ denotes the traceless part of a $(0,2)-$tensor. Moreover, we have the following asymptotic expansion on $N\setminus K$:
    \begin{equation}\label{o0expansion}
        (\mathscr{L}_{\nabla u}g_b)^{\circ}_{ij}=-4m\mathrm{Area}_{g_E}\left(S^{m-1}/\Gamma\right)^{-1}\mathcal{V}\left(\frac{mx^ix^j}{r^{m+2}}-\frac{\delta_{ij}}{r^m}\right)-2mW^\infty_{ikjl}\frac{x^kx^l}{r^{m+2}}+O(r^{-m-1+\varepsilon}).
    \end{equation}
    Here, $\mathcal{V}$ is the renormalized volume of $(N,g_b)$, and $W^{\infty}$ is the asymptotic Weyl curvature of $(N,g_b)$. We refer the reader back to Definition \ref{WVdef} for their definitions.
\end{proposition}
\begin{proof}
    First of all,
\[
(\mathscr{L}_{\nabla u} g_b)^{\circ} = 2(\operatorname{Hess}_{g_b} u)^{\circ} = 2\operatorname{Hess}_{g_b} u - \frac{2}{m}(\Delta_{g_b} u)g_b = 2\operatorname{Hess}_{g_b} u - 4g_b.
\]

By diffeomorphism invariance of $\mathrm{Ric}$, we know $d\Ric_{g_b} (\mathscr{L}_{\nabla u} g_b) = 0$, hence:
\[
d\mathrm{Ric}_{g_b}\bigl((\mathscr{L}_{\nabla u} g_b)^{\circ}\bigr) = d\mathrm{Ric}_{g_b}(\mathscr{L}_{\nabla u}g_b - 4g_b) = 0.
\]
By the Ricci identity
\[
\Delta\nabla u - \nabla\Delta u = \mathrm{Ric}(\nabla u),
\]
we know that
\[
\divg_{g_b}\bigl((\mathscr{L}_{\nabla u} g_b)^{\circ}\bigr) = 2\divg_{g_b} (\operatorname{Hess}_{g_b} u) = 2\Delta\nabla u = 0.
\]
Putting $d\mathrm{Ric}_{g_b}\bigl((\mathscr{L}_{\nabla u} g_b)^{\circ}\bigr) = 0$, $\divg_{g_b}\bigl((\mathscr{L}_{\nabla u} g_b)^{\circ}\bigr) = 0$, $\operatorname{tr}_{g_b}\bigl((\mathscr{L}_{\nabla u} g_b)^{\circ}\bigr) = 0$ together, we get $P_{g_b}\bigl((\mathscr{L}_{\nabla u} g_b)^{\circ}\bigr) = 0$.

The asymptotic expansion \eqref{o0expansion} follows directly from expansions
\[
g_{b,ij} = \delta_{ij} + W_{ikjl}^{\infty} \frac{x^kx^l}{r^{m+2}} + O(r^{-m-1}),
\]
\[
u = r^2 + \frac{2m}{2 - m} \, \mathrm{Area}_{g_E}\left(S^{m-1}/\Gamma\right)^{-1}\mathcal{V}\cdot r^{2 - m} + O(r^{1 - m+\varepsilon}).
\]
In particular, expansion \eqref{o0expansion} tells that $(\mathscr{L}_{\nabla u} g_b)^{\circ}$ is indeed decaying, so $(\mathscr{L}_{\nabla u} g_b)^{\circ} \in \SO(g_b)$.
\end{proof}

\subsection{On compact orbifolds}
For a compact Einstein orbifold $(M_0,g_0)$, we denote $\mathscr{M}$ to be the space of smooth metrics on $M_0$, and $\mathscr{M}_1$ to be the space of smooth metrics with the same volume as $g_0$. Then
\[T_{g_0}\mathscr{M}_1=\left\{h\in \Gamma(Sym^2T^*M_0)\left|\int_{M_0}(\mathrm{tr}_{g_0}h)dV_{g_0}=0\right.\right\},\]
and
\[T_{g_0}\mathscr{M}=T_{g_0}\mathscr{M}_1\oplus\R\langle g_0\rangle .\]
In other words, we split infinitesimal deformations into two parts, one is volume-preserving and the other consists of rescaling of the metric. Since a rescaling of Einstein metric is still Einstein, $\lambda g_0$ is indeed an IED. It is well-know that, the space of "volume-preserving IED in Bianchi gauge" coincides with $\ker\left(P_{g_0}|_{T_{g_0}\mathscr{M}_1}\right)$. For instance, one can see {\cite[equation (12.28)]{Besse}} for this fact.

\begin{notation}\label{IED0}
    We define $\SO(g_0)=\ker(P_{g_0})\cap T_{g_0}\mathscr{M}_1$, which consists of only volume-preserving IED of $g_0$. In our notation, the space of all IEDs of $g_0$ is $\SO(g_0)\oplus\R\langle g_0\rangle$.
\end{notation}

\subsection{Approximate IED on glued manifold}
Since $P_{g_b}$ and $P_{g_0}$ are both elliptic, by choosing weighted function spaces in non-exceptional weights, we know they are Fredholm operators. Hence $\SO(g_b)$ and $\SO(g_0)$ are finite dimensional. 

We glue $\SO(g_b)$ and $\SO(g_0)$ to approximate IED space on $(M,g_t)$. To fit in the estimate under decoupling norms $\ctab$, we need to decouple elements in $\SO(g_0)$ near the orbifold singularity, as we will describe now.

For the orbifold point $p\in M_0$, we choose a cutoff function $\chi_{B(p,\delta)}$ supported in $d(p,\cdot)<2\delta$ and equals to $1$ in $d(p,\cdot)<\delta$. For $O\in\SO(g_0)$, we decompose $O=O_*+\chi_{B(p,\delta)}O_A$, where $O_*\in C^{2,\alpha}_{\beta}(g_0)$ and $O_A$ is a constant symmetric 2-tensor, chosen such that
\[\|O_*\|_{C^{2,\alpha}_{\beta}(g_0)}+|O_A|_{g_E}=\inf\left\{\|h_*\|_{C^{2,\alpha}_{\beta}(g_0)}+|h_A|_{g_E}\right\}\]
where the infimum is taken among $O=h_*+\chi_{A(t,\delta)}h_A$, $h_*\in C^{2,\alpha}_{\beta}(g_0)$ and $h_A$ is a constant symmetric 2-tensor on $\R^m/\Gamma$. Thanks to the cufoff functions in Notation \ref{cut-on-neck}, we define a symmetric 2-tensor on $M$ as
\[\bar{O}_t=(1-\chi_t)O_*+\chi_{A(t,\delta)}O_A.\]
Also, for $o\in\SO(g_b)$, we define a symmetric 2-tensor on $M$ as
\[\bar{o}_t=t^2\chi_to.\]

\begin{definition}\label{AIED}
    Suppose $\SO(g_b)$ is spanned by $\{o_i\}_i$ and $\SO(g_0)$ is spanned by $\{O_j\}_j$, we define the space of \textbf{approximate IED on $\gtd$} as the finite dimensional vector space spanned by the following tensors:
    \begin{itemize}
        \item $\bar{o}_{i,t}+\bar{O}_{j,t}$
        \item $\gtd$
    \end{itemize}
    and we denote the approximate IED space as $\tilde{\SO}(\gtd)$. Also, we denote $\tilde{\SO}(\gtd)^{\perp}$ to be its $L^2(\gtd)-$orthogonal complement.
\end{definition}

In definition \ref{AIED}, the first type of tensors consist of gluing $\SO(g_b)$ with volume-preserving IED of $g_0$, and the second type of tensors are true IED of $\gtd$ corresponding to rescaling of the metric.
\begin{notation}\label{pi}
    We define the projection $\pi$ onto $\TO(\gtd)$ as
    \[\pi(h):=\sum_i(h,\tilde{o}_i)_{L^2(\gtd)}\tilde{o}_i,\]
    where $\{\tilde{o}_i\}$ is any orthonormal basis of $\TO(\gtd)$. Also, we denote $\pi^\perp:=1-\pi$ to be the projection onto $\TO(\gtd)^\perp$.
\end{notation}

To show that $\TO(\gtd)$ is sufficient to approximate IED space of $g_t$, we estimate the elliptic opertator $P_{g_t}$, after ruling out this approximate kernel. 
\begin{proposition}\label{P-1}
    There exist constants $0<\delta_0\ll1$ such that, for all $0<\delta<\delta_0$ and $t=\frac{\delta^2}{16}$, the restriction (acting on $(0,2)$-tensor fields)
    \[\pi^\perp\circ P_{g_t}:\ctab(\gtd)\cap\tilde{\SO}(g_{t})^{\perp}\to \cab(\gtd)\cap\TO(\gtd)^\perp\]
    is invertible. Moreover, there exists $C>0$ such that
    \begin{equation}\label{equ3.5}
        \|s\|_{\ctab}\leq C\|\pi^\perp\circ P_{g_{t}}s\|_{\cab}, \quad\forall s\in\ctab(\gtd) \cap\tilde{\SO}(\gtd)^{\perp}.
    \end{equation}
\end{proposition}
Again, the proof of this proposition can be obtained by verbatim adaptation of the original argument given by Ozuch. We refer the reader to {\cite[Proposition 4.9]{Ozu2}} for details.

\section{Detecting Einstein metrics in Bianchi gauge}
In this section, we denote $\ctab(\gtd)$ and $\cab(\gtd)$ as symmetric (0,2)-tensor fields with finite weighted (decoupling) H\"older norms. Consider the following map
\[\Phi=\Phi_{\gtd}:\ctab(\gtd)\to\cab(\gtd),\]
\[
\Phi(h) := \Ric(g)- \frac{1}{m }\left( \frac{1}{\V(g)} \int_M \s_g d\V_{g} \right) g + \delta_{\gtd}^* B_{g_t}(g),
\]
where \( g := \gtd + h \). So finding an Einstein metric $g$ near $g_t$ in Bianchi gauge, is equivalent to solving
\[\Phi_{g_t}(g-g_t)=0.\]
Combining Theorem \ref{neckana} and Proposition \ref{Slice-theorem}, we get 
\begin{proposition}\label{Phi=0}
    Suppose $(M_0,g_0)$ is a compact Einstein orbifold with negative Einstein constant, and it admits a resolution sequence. 
    
    Then  there exist sequences of parameters $0<t_i=\frac{\delta_i^2}{16}\ll1$ with $\lim_{i\to\infty}t_i=\lim_{i\to\infty}\delta_i=0$, such that the gluing metrics $\gtidi$ approximate Einstein metrics in Bianchi gauge. 
    
    More precisely, for any $\varepsilon>0$, there exists $i_0\gg 1$, such that for each $i>i_0$ there is an Einstein metric $g_i$ with
    \[B_{\gtidi}g_i=0,\quad\|g_i-\gtidi\|_{\ctab(\gtidi)}< \varepsilon.\]
    In particular, $\Phi_{\gtidi}(g_i-g_{t_i})=0$.
\end{proposition}

\subsection{The linearized map} 
For now, we fix a glued metric $g_t$ with sufficiently small $t$, and estimate the inverse of \( d\Phi_0 \) restricted on \( \tilde{\SO}(\gtd)^{\perp} \). First we calculate \( d\Phi_0 \). We have
\[
d\Ric_{\gtd} (s) = P_{\gtd} s + \frac{1}{2} \left( \Ric_{\gtd} \circ s + s \circ \Ric_{\gtd} \right) - \delta_{\gtd}^* B_{\gtd} s,
\]
where \( (\Ric_g\circ s)_{ij} := g^{kl} \Ric_{ik} s_{lj} \), \( (s \circ \Ric_g)_{ij} := g^{kl} s_{ik} \Ric_{lj} \). Also,
\[d\s_{g_t}(s)=-\langle\Ric_{g_t},s\rangle+\delta_{g_t}^2 s-\Delta(\tr_{g_t} s),\]
\[
d \left( \V(\cdot) \right)_{\gtd} (s) = \frac{1}{2} \int_M (\tr_{\gtd}s) d\V_{\gtd},
\]
\[
d \left( \int_M \s(\cdot) d\V_{\cdot} \right)_{\gtd} (s) = \int_M \left\langle -\Ric_{\gtd} + \frac{\s_{\gtd}}{2} \gtd, s \right\rangle d\V_{\gtd}.
\]
Combining these together, we have
\[
\begin{split}
d\Phi_0(s) =& P_{\gtd} s + \frac{1}{2} \left( \Ric_{\gtd} \circ s + s \circ \Ric_{\gtd} \right)-\frac{1}{m} \left( \frac{1}{ \V(\gtd)} \int_M \s_{\gtd} d\V_{\gtd} \right) s \\
& +\frac{1}{2m} \frac{\gtd}{ \V(\gtd)^2}  \left( \int_M \tr_{\gtd}s d\V_{\gtd}\right) \left( \int_M \s_{\gtd} d\V_{\gtd} \right)\\
&-\frac{1}{m} \frac{\gtd}{ \V(\gtd)}\int_M \left\langle -\Ric_{\gtd} + \frac{\s_{\gtd}}{2} \gtd, s \right\rangle d\V_{\gtd}.
\end{split}
\]
Restricted to $\TO(\gtd)^\perp$, in particular, 
\[0=\int_M\langle s,g_t\rangle d\V_{g_t}=\int_M(\tr_{g_t}s)d\V_{g_t},\]
we have
\[
\begin{split}
d\Phi_0(s) =& P_{\gtd} s + \frac{1}{2} \left( \Ric_{\gtd} \circ s + s \circ \Ric_{\gtd} \right)-\frac{1}{m} \left( \frac{1}{ \V(\gtd)} \int_M \s_{\gtd} d\V_{\gtd} \right) s \\
&- \frac{1}{m} \frac{\gtd}{ \V(\gtd)}  \int_M \left(\left\langle -\Ric_{\gtd} + \frac{\s_{\gtd}}{m} \gtd, s \right\rangle+\frac{m-2}{2m}\s_{g_t}(\tr_{g_t}s)\right) d\V_{\gtd}.
\end{split}
\]

\begin{lemma}\label{3.5}
For each $0<t=\frac{\delta^2}{16}\ll1$, there exists a constant \( \theta(t) > 0 \), with \( \lim_{t \to 0} \theta(t) = 0 \), such that
\begin{equation}\label{equ3.3}
    \| (\pi^\perp\circ d\Phi_0 -\pi^\perp\circ P_{\gtd}) s \|_{\cab(\gtd)} \leq \theta(t) \| s \|_{\ctab(\gtd)}, \quad \forall s \in \ctab(\gtd) \cap \tilde{\SO}(\gtd)^{\perp}.
\end{equation}
\end{lemma}
\begin{proof}
    We have $\Ric_{g_t}=\mu g_t$ on $M_0^\delta=\{r_t\geq 10\sqrt{t}\}$, where $\mu$ is the Einstein constant of $g_0$, and $|\Ric_{g_t}-\mu g_t|$ is bounded on the rest of $M$. So
    \[\|\Ric_{g_t}-\mu g_t\|_{\cab}\leq Ct^{4-\beta/2}.\]
    Hence we have the following approximations up to a positive power of $t$:
    \[\frac{1}{2} \left( \Ric_{\gtd} \circ s + s \circ \Ric_{\gtd} \right)-\frac{1}{m} \left( \frac{1}{ \V(\gtd)} \int_M \s_{\gtd} d\V_{\gtd} \right) s\sim \mu s-\frac{1}{m}m\mu s=0,\]
    \[\int_M\left\langle -\Ric_{\gtd} + \frac{\s_{\gtd}}{m} \gtd, s \right\rangle d\V_{g_t}\sim 0,\]
    \[\int_M\frac{m-2}{2m}\s_{g_t}(\tr_{g_t}s)d\V_{g_t}\sim \frac{m-2}{2m}m\mu\int_M(\tr_{g_t}s)d\V_{g_t}=0.\]
    So the conclusion holds.
\end{proof}
As a corollary of Proposition \ref{P-1} and Lemma \ref{3.5}, by triangle inequality we have
\begin{corollary}\label{linear-estimate}
    There exists a constant $0<\delta_0\ll1$ such that, for all $0<\delta<\delta_0$ and $t=\frac{\delta^2}{16}$, the restriction
    \[\pi^\perp\circ d\Phi_0:\ctab(g_t)\cap\TO(g_t)^\perp\to \cab\cap\TO(g_t)^\perp\]
    is invertible. Moreover, there exists $C>0$ such that
    \begin{equation}\label{equ3.10}
        \|s\|_{\ctab}\leq C\|\pi^\perp\circ d\Phi_0(s)\|_{\cab},\quad\forall s\in\ctab(g_t)\cap\TO(g_t)^\perp.
    \end{equation}
\end{corollary}

\subsection{Local moduli space around the glued manifold}
Now we give a description on local structure of $\M(m,\Lambda,D,V,E)$ near the glued manifold $(M,g_t)$, in the same spirit of {\cite[Theorem 0.9]{Koiso}}, and follows directly from our previous estimate in Corollary \ref{linear-estimate}. Since our main theorem \ref{Maintheorem} does not rely on this precise description, we omit the lengthy details, and only give a sketch of its proof. 
\begin{theorem}\label{local-moduli}
    There exist $\delta_0>0$ and $\varepsilon>0$, such that for any $0<t=\frac{\delta^2}{16}<\frac{\delta_0^2}{16}$ and for any $v\in\TO(g_t)$ with $\|v\|_{\ctab(g_t)}<\varepsilon$, there exists a unique $h_v\in\ctab(g_t)$ such that
    \begin{itemize}
        \item $\Phi_{g_t}(h_v)\in\TO(g_t)$;
        \item $\|h_v\|_{\ctab(g_t)}<2\varepsilon$;
        \item $h_v-v\in\TO(g_t)^\perp$.
    \end{itemize}
    
    As a corollary, the moduli space of Einstein metrics near $g_t$ is a subvariety of $\TO(g_t)$, which is in one-to-one correspondence with the zero set of
    \[v\mapsto h_v\mapsto\Phi_{g_t}(h_v).\]
\end{theorem}
\begin{proof}[Sketch of proof]
    We first prove the theorem for $v=0$. Consider the restriction
    \[\pi^\perp\circ \Phi_{g_t}:\ctab(g_t)\cap\TO(g_t)^\perp\to \cab\cap\TO(g_t)^\perp.\]
    By Corollary \ref{linear-estimate}, its differential at $0$ is invertible, with bounded inverse. To apply the implicit function theorem (Lemma \ref{IFT}), we need the estimates for $\pi^\perp\circ \Phi_{g_t}(0)$ and the quadratic term $Q$, which follow from direct calculations when $t$ is small enough.

    For general $v$ with $\|v\|_{\ctab(g_t)}<\varepsilon$, we replace $g_t$ by $g_t+v$, and solve $\pi^\perp\circ\Phi_{g_t+v}(h)=0$ using Lemma \ref{IFT}.
\end{proof}

\section{Refined gluing}
In this section, we modify the metrics on both Ricci-flat ALE manifold and compact orbifold, to get a refined gluing. The refinement is done in two aspects:
\begin{itemize}
    \item As illustrated by Theorem \ref{local-moduli}, for each Einstein metric $g$ near $g_t$, there exists a "projection of $g$ on $\TO(g_t)$", that is, $v\in\TO(g_t)$ such that $g-(g_t+v)\perp\TO(g_t)$. Hence, in order to get a better approximation of $g$, we would first deform the compact orbifold and the blow-up limit, then glue the deformed metrics together. It is worth noting that Einstein deformations are often obstructed, so we only get obstructed Einstein metrics on both parts.
    \item Motivated by "find back the Ricci curvature on blow-up limits", we solve a linearized Ricci equation on Ricci-flat ALE manifolds. Again, this equation does not have a solution in general, but we can formulate its obstruction explicitly.
\end{itemize}

\subsection{Einstein deformation on blow-up limits and compact orbifolds}
First, we give a version of implicit function theorem, when the linearized operator is neither injective nor surjective. To set up the situation, we consider Banach spaces $X,Y$, an open neighborhood $U$ of $0\in X$ and a smooth Fredholm map $F:U\to X$. In particular, $dF_0:X\to Y$ is a Fredholm operator. Now we decompose $X,Y$ as
\[X=X_1\oplus X_2,\quad Y=Y_1\oplus Y_2,\]
where $X_2=\ker(dF_0)$, $Y_1=\mathrm{Im}(dF_0)$ and $Y_2\cong\mathrm{coker}(dF_0)$. Hence the restriction
\[dF_0:X_1\to Y_1\]
is an isomorphism. Since $X_2,Y_2$ are finite dimensional, there are projections
\[\Pi_{X_i}:X\to X_i,\quad \Pi_{Y_i}:Y\to Y_i,\quad (i=1,2).\]
\begin{lemma}\label{LSreduction}
    Let $X,Y,F$ satisfy the assumptions above, and let $Q := F - F(0) - dF_0$. Suppose there are constants $q, r > 0$, $c\geq 1$ such that:
    \begin{enumerate}
        \item[(1)] $\|Q(x) - Q(y)\| \leq q\|x - y\|(\|x\| + \|y\|)$, $\forall x, y \in B(0, r) \subset X$;
        \item[(2)] $\|x\|\leq c\|dF_0(x)\|$, $\forall x\in X_1$;
        \item[(3)] $r < \min\left\{1, \frac{1}{2qc}\right\}$ and $\|F(0)\| \leq \frac{r}{4c}$.
    \end{enumerate}
    Then for each $v\in X_2$, $\|v\|\leq \frac{r}{4c}$, there exists a unique $x_v\in X$ such that
    \begin{itemize}
        \item $F(x_v)\in Y_2$,
        \item $x_v-v\in X_1$,
        \item $\|x_v\|_X+\|F(x_v)\|_Y<r$.
    \end{itemize}
    Moreover, the differential of the map $v\mapsto x_v$ at $0$ is the inclusion map $X_2\hookrightarrow X$.
\end{lemma}
\begin{proof}
    For any $v\in X_2\cap B\left(0,\frac{r}{4c}\right)$, we define a new map
    \[H:X\oplus Y_2\to Y\oplus X_2,\quad H(x,w):=(F(x)-w,\Pi_{X_2}(x-v)).\]
    Then $H(x,w)=0$ if and only if $F(x)=w\in Y_2$ and $x-v\in X_1$.
    
    In view of decompositions of $X$ and $Y$, we write the differential of $H$ as
    \[dH_0:X_1\oplus X_2\oplus Y_2\to Y_1\oplus Y_2\oplus X_2,\]
    and $dH_0(x_1,x_2,w)=(dF_0(x_1+x_2),-w,\Pi_{X_2}(x_1+x_2))=(dF_0(x_1),-w,x_2)$. Also, $\|H(0)\|=\|(F(0),-v)\|=\|F(0)\|+\|v\|\leq\frac{r}{2c}$. So Lemma \ref{IFT} applies to the map $H$, and the conclusion for $x_v$ holds.

    For the "moreover" part, we compute the differential of $v\mapsto x_v$ by implicit function theorem. We define 
    \[\tilde{H}:X\oplus Y_2\oplus X_2\to Y\oplus X_2,\quad \tilde{H}(x,w,v):=(F(x)-w,\Pi_{X_2}(x)-v).\]
    The arguments above shows that for each small $v$, there exists a unique solution $(x_v,w_v)$ near 0, such that $\tilde{H}(x_v,w_v,v)=0$. Now we compute the partial differentials
    \[\left.\frac{\partial\tilde{H}}{\partial (x,w)}\right|_{0}:X\oplus Y_2\to Y\oplus X_2,\]
    which equals $dH_0$, and
    \[\left.\frac{\partial\tilde{H}}{\partial v}\right|_{0}:X_2\to Y\oplus X_2,\quad v\mapsto(0,-v).\]
    Hence
    \[\left.\frac{d(x_v,w_v)}{d v}\right|_{0}=-\left(\left.\frac{\partial\tilde{H}}{\partial (x,w)}\right|_{0}\right)^{-1}\circ\left.\frac{\partial\tilde{H}}{\partial v}\right|_{0}:X_2\to X\oplus Y_2,\quad v\mapsto(v,0).\]
    In particular, $\left.\frac{dx_v}{d v}\right|_{0}:X_2\to X$ is the inclusion.
\end{proof}
\begin{remark}
    The lemma above and Theorem \ref{local-moduli} address similar problems, both concerning the implicit function theorem for Fredholm maps. However, the situation of Theorem \ref{local-moduli} is more subtle, since $\TO(g_t)$ is not the exact kernel of $d\Phi_0$. By contrast, the advantage of Lemma \ref{LSreduction} lies in its "moreover" part, which calculates the differential of $v\mapsto x_v$. This will be used later in Lemma \ref{lem4.12}.
\end{remark}
Second, we describe obstructed Einstein deformations on Ricci-flat ALE manifolds and compact Einstein orbifolds. We denote $C^{2,\alpha}_\beta(g_b)$ and $C^{2,\alpha}(g_0)$ as symmetric $(0,2)-$tensor fields with finite (weighted) H\"older norms.

For ALE Ricci-flat manifold $(N,g_b)$, we define
\[\Psi:C^{2,\alpha}_\beta(g_b)\to C^{0,\alpha}_{\beta+2}(g_b), h\mapsto \Ric(g_b+h)+\delta_{g_b}^*B_{g_b}(g_b+h).\]
In fact, $\Psi$ is defined only on a neighbourhood of 0, but this will not cause any problems.
\begin{proposition}\label{bdeform}
    There exist $\eta,\theta>0$ such that, for any $v\in\ker(d\Psi(0))$, $\|v\|\leq\eta$, there exists a unique $h_{v}\in C^{2,\alpha}_\beta(g_b)$ satisfying
    \begin{itemize}
        \item[(i)] $\Psi(h_{v})\in\coker(d\Psi_0)$;
        \item[(ii)] $h_{v}-v$ is $L^2(g_b)-$orthogonal to $\ker(d\Psi_0)$;
        \item[(iii)] $\|h_{v}\|+\|v\|<\theta$.
    \end{itemize}
    Moreover, the map $v\mapsto h_v$ is $C^1$. In particular, as $\|v\|_{C^{2,\alpha}_\beta(g_b)}\to0$, we have $\|h_v\|_{C^{2,\alpha}_\beta(g_b)}\to 0$.
\end{proposition}
\begin{proof}
    Directly apply Lemma \ref{LSreduction} with $F=\Psi$, $X=C^{2,\alpha}_\beta(g_b)$, $X_1=\ker(d\Psi_0)^\perp$, $Y=C^{0,\alpha}_{\beta+2}(g_b)$.
\end{proof}
\begin{remark}\label{rmk4.5}
    One can compute $d\Psi_0=P_{g_b}$, hence 
    \[\coker(d\Psi_0)=\ker(d\Psi_0)=\SO(g_b).\]
    Since there is no exceptional value in $(2-m,0)$, in this lemma, we can take $\beta=m-2-\varepsilon$ for any $0<\varepsilon<1$. 
\end{remark}

For compact Einstein orbifold $(M_0,g_0)$, we define
\[\Psi':C^{2,\alpha}(g_0)\to C^{0,\alpha}(g_0), h\mapsto \Ric(g) - \left( \frac{1}{m \V(g)} \int_M \s_g d\V_{g} \right) g + \delta_{g_0}^* B_{g_0} (g),\]
where $g=g_0+h$. In fact, $\Psi'$ is defined only on a neighbourhood of 0, but this will not cause any problems.
\begin{proposition}\label{0deform}
    There exist $\eta,\theta>0$ such that, for any $w\in\ker(d\Psi'(0))$, $\|w\|\leq\eta$, there exists a unique $h_{w}\in C^{2,\alpha}(g_0)$ satisfying
    \begin{itemize}
        \item[(i)] $\Psi'(h_{w})\in\coker(d\Psi'_0)$;
        \item[(ii)] $h_{w}-w$ is $L^2(g_b)-$orthogonal to $\ker(d\Psi'_0)$;
        \item[(iii)] $\|h_{w}\|+\|w\|<\theta$.
    \end{itemize}
\end{proposition}
\begin{proof}
    Directly apply Lemma \ref{LSreduction} with $F=\Psi'$, $X=C^{2,\alpha}(g_0)$, $X_1=\ker(d\Psi'_0)^\perp$, $Y=C^{0,\alpha}(g_0)$.
\end{proof}
\begin{remark}
    Since $\Ric(g_0)=\mu g_0$, one can compute that
    \[d\Psi'_0(s)=P_{g_0}s+\frac{\mu}{m\V(g_0)}\left(\int_{M_0}\tr_{g_0}sd\V_{g_0}\right)g_0.\]
    Recall we defined $\SO(g_0)=\ker(P_{g_0})\cap T_{g_0}\mathscr{M}_1$ in section 4.2, so 
    \[\coker(d\Psi'_0)=\ker(d\Psi'_0)=\SO(g_0)\oplus\R\langle g_0\rangle.\]
\end{remark}

\subsection{Linearized Ricci equation on blow-up limits}\label{sec4.3}
Suppose $\Ric(g_0)=\mu g_0$. In this subsection, we wish to find a symmetric $(0,2)-$tensor field $h_2$ on $(N,g_b)$ such that
\[\begin{cases}
    \Ric(g_b+t^2h_2)=t^2\mu(g_b+t^2h_2)\text{ for each small }t,\\
    B_{g_b}h_2=0.
\end{cases}\]
Taking derivative with respect to $s=t^2$, at $s=0$ we must have
\[\begin{cases}
    d\Ric_{g_b}h_2=\mu g_b\\
    B_{g_b}h_2=0
\end{cases}\Longleftrightarrow\begin{cases}
    P_{g_b}h_2=\mu g_b\\
    B_{g_b}h_2=0
\end{cases}.\]

First of all, we find a candidate for the asymptotic behaviour of $h_2$. 
\begin{lemma}\label{lemmaH}
    We assume that in geodesic coordinates around the singular point \( p \in (M_0,g_0) \), the metric is expanded as
    \[g_0=g_E- \frac{1}{3} R_{ikjl}(0) x^k x^ldx^idx^j+O(|x|^3).\]
    There exists another coordinate system around $p$, such that we have an expansion
    \[
    g_0 = g_E + H + O(|x|^3),
    \]
    where
    \begin{equation}\label{H}
        H_{ij} := H\left(\frac{\partial}{\partial x^i},\frac{\partial}{\partial x^j}\right)=- \frac{1}{3} R_{ikjl}(0) x^k x^l - \frac{2\mu}{3(m+2)}\left(|x|^2 \delta_{ij} + 2x^i x^j\right).
    \end{equation}

    In particular, \( H \) satisfies
    \begin{itemize}
        \item[(i)] \( H_{ij} \) is a homogeneous polynomial of degree 2;
        \item[(ii)] \( B_{g_E}(H) = 0 \);
        \item[(iii)] \(\frac{1}{2}\nabla^*_{g_E}\nabla_{g_E} H = \mu g_{E} \).
    \end{itemize}
\end{lemma}
\begin{proof}
    In geodesic coordinates around \( p \),
    \[
    (g_0)_{ij} = \delta_{ij} - \frac{1}{3} R_{ikjl}(0) x^k x^l + O(|x|^3),
    \]
    and $R_{ikjk}(0)=\mu\delta_{ij}$. Consider a coordinate change
    \[
    \varphi(x)^j = x^j + c |x|^2 x^j, \quad \text{where } c = -\frac{\mu}{3(m+2)}.
    \]
    Then
    \[
    (\varphi^* g_0)_{ij} = \delta_{ij} - \frac{1}{3} R_{ikjl}(0) x^k x^l + 2c\left(|x|^2 \delta_{ij} + 2x^i x^j\right) + O(|x|^3).
    \]

    In this coordinate, the quadratic term $H$ is given by
    \[
    \begin{aligned}
    H_{ij} &= - \frac{1}{3} R_{ikjl}(0) x^k x^l + 2c\left(|x|^2 \delta_{ij} + 2x^i x^j\right)\\
    &= - \frac{1}{3} R_{ikjl}(0) x^k x^l - \frac{2\mu}{3(m+2)}\left(|x|^2 \delta_{ij} + 2x^i x^j\right) .
    \end{aligned}
    \]
    We have \( H_{ij} \) is homogeneous of degree 2,
    \[
    B_{g_E}(H) = -\partial_i H_{ij} + \frac{1}{2} \partial_j H_{ii} = 0,
    \]
    \[
    \left(\frac{1}{2}\nabla^*_{g_E}\nabla_{g_E} H\right)_{ij} =- \frac{1}{2} \partial_k \partial_k H_{ij} = \mu \delta_{ij},
    \]
    so such a coordinate is what we want.
\end{proof}
In the following, we use \( H = H_{ij} dx^i dx^j \) given by the lemma above, which can be viewed as a tensor field on \( \R^m/\Gamma\). Next, we fix a cutoff function $\bar{\chi}$ on $N$, such that $\mathrm{supp}(\bar{\chi})\subset\{r_b>2R\}$ and $\bar{\chi}\equiv1$ on $\{r_b\geq 3R\}$, and solve
\[\begin{cases}
    P_{g_b}h_2=\mu g_b\\
    B_{g_b}h_2=0\\
    h_2 = \bar{\chi} H + O(r^{2-m+\varepsilon}), \ \forall \varepsilon > 0
\end{cases}.\]
A priori, this system may not have a solution, whose obstruction is given by $\coker(P_b)$.
\begin{proposition}\label{h2}
    The system 
\begin{equation}\label{curvature-perturb}
    \begin{cases}
        P_{g_b}h=\mu g_b+ \sum \lambda_j o_j\\
        B_{g_b}h=0\\
        h = \bar{\chi} H + O(r^{2-m+\varepsilon}), \ \forall \varepsilon > 0
\end{cases}
\end{equation}
always has a solution $h$ with 
\begin{equation}\label{lambda_j}
    \begin{aligned}
        \lambda_j =-\frac{m+2}{2}\lim_{r \to \infty} \int_{S_r / \Gamma}  \frac{\langle H, o_j \rangle}{r}   d\sigma_{S_r / \Gamma},
    \end{aligned}
    \end{equation}
where $\{ o_i \}$ is an orthonormal basis of $\coker(P_{g_b})=\SO(g_b)$. The solution is determined modulo $\ker(P_{g_b})=\SO(g_b)$. We can fix a particular solution $h_2$ by requiring $h_2 \perp \SO(g_b)$ with respect to $L^2$ inner product.
\end{proposition}

\begin{proof}
    \textbf{Step 1: Verifying the Bianchi gauge.} Suppose $h$ is a solution of 
    \begin{equation}\label{equ4.4}
        \begin{cases}
            P_{g_b}h=\mu g_b+ \sum \lambda_j o_j\\
            h = \bar{\chi} H + O(r^{2-m+\varepsilon}), \ \forall \varepsilon > 0
        \end{cases}
    \end{equation}
    for some $\{\lambda_i\}$, then we show that $B_{g_b}h=0$. By Proposition \ref{IEDb}(2), 
    \[B_{g_b}P_{g_b}h=\mu B_{g_b}(g_b)+\sum\lambda_j B_{g_b}o_j=0.\]
    By Ricci identity, 
    \[0=B_{g_b}P_{g_b}h=\frac{1}{2}\nabla^*_{g_b}\nabla_{g_b}B_{g_b}h.\]
    Also, $h=\bar{\chi}H+O(r^{2-m+\varepsilon})$ implies $B_{g_b}h=O(r^{1-m+\varepsilon})$, a simple integration by parts shows that $B_{g_b}h=0$. So it remains to solve \eqref{equ4.4}.

    \textbf{Step 2: Solving \eqref{equ4.4}.} Let \(k = h-\bar{\chi} H\), then \(k\in C^{2,\alpha}_{m-2-\varepsilon}(g_b)\) and
    \begin{equation}\label{equ4.5}
        P_{g_b}(k)=\mu g_{b}-P_{g_b}(\bar{\chi} H)+ \sum \lambda_j o_j.
    \end{equation}
    
    First of all, we show that $\mu g_{b}-P_{g_b}(\bar{\chi} H)=O(r^{-m})$. In fact, 
    \begin{align*}
        P_{g_b}(\bar{\chi} H)&=P_{g_b}(\bar{\chi} H)+P_{g_b}(\bar{\chi}\delta^*_{g_E}V)=P_{g_b}(H)+P_{g_E}(\delta^*_{g_E}V)+O(r^{-m})\\
        &=\frac{1}{2}\nabla^*_{g_b}\nabla_{g_b}H-\mathring{R}_{g_b}(H)+O(r^{-m})\\
        &=\mu g_E+O(r^{-m}),
    \end{align*}
    and $g_b=g_E+O(r^{-m})$.

    Now we solve \eqref{equ4.5}, where the right hand side is $O(r^{-m})$. Such \(k\) exists if and only if the right hand side lies in $\mathrm{Im}(P_{g_b})=\coker(P_{g_b})^{\perp}$, or equivalently, 
    \[\lambda_j:=-(\mu g_{b}-P_{g_b}(\bar{\chi} H),o_j)_{L^2(g_b)}.\]

    \textbf{Step 3: Formulating $\lambda_j$.} Since each $o_j$ is trace-less,
    \begin{align*}
        \int_{r\leq R}\langle\mu g_{b}-P_{g_b}(\bar{\chi} H),o_{j}\rangle dV&=\int_{r\leq R}\langle - P_{g_b}(\bar{\chi} H),o_j\rangle dV\\
        &=-\frac{1}{2}\int_{r\leq R}\langle\nabla^{*}\nabla(\bar{\chi} H),o_j\rangle dV+\int_{r \leq R}\langle\mathring{R}(\bar{\chi} H),o_j\rangle dV.
    \end{align*}
    Using integration by parts, 
    \begin{align*}
    &\int_{r\leq R}\langle\nabla^{*}\nabla(\bar{\chi} H),o_j\rangle dV\\
    =&-\int_{r = R}\langle\nabla_{\nu}(\bar{\chi} H),o_j\rangle d\sigma+\int_{r\leq R}\langle\nabla(\bar{\chi} H),\nabla o_j\rangle dV\\
    =&-\int_{r = R}\langle\nabla_{\nu}(\bar{\chi} H),o_j\rangle d\sigma+\int_{r= R}\langle\bar{\chi} H,\nabla_{\nu} o_j\rangle d\sigma+\int_{r\leq R}\langle\bar{\chi} H,\nabla^*\nabla o_j\rangle dV.
    \end{align*}
    Note that \(P_{g_b}(o_j) = 0\). We have
    \[\int_{r\leq R}\langle\mu g_{b}-P_{g_b}(\bar{\chi} H),o_j\rangle dV=\frac{1}{2}\int_{r = R}\langle\nabla_{\nu}(\bar{\chi} H),o_j\rangle d\sigma-\frac{1}{2}\int_{r= R}\langle\bar{\chi} H,\nabla_{\nu} o_j\rangle d\sigma.\]

    Later, we will take the limit \(R\rightarrow\infty\), so now we use $\sim$ to denote that two quantities differ by a decaying term which will eventually tend to 0. We have 
    \[\nu\sim\partial_r,\quad\nabla_{\nu}(\bar{\chi} H)\sim\partial_{r}H\sim\frac{2}{r}H,\quad\nabla_{\nu}o_j\sim\partial_{r}o_j\sim\frac{-m}{r}o_j.\]
    So
    \begin{align*}
        \lambda_{i}&=-(\mu g_{b}-P_{g_b}(\bar{\chi} H),o_j)_{L^{2}(g_{b})}\\
        &=-\lim_{R\rightarrow\infty}\int_{r\leq R}\langle\mu g_{b}-P_{g_b}(\bar{\chi} H),o_j\rangle dV \\
        &= -\frac{2+m}{2}\lim_{r\rightarrow\infty}\int_{S_{r}}\frac{\langle H,o_j\rangle}{r}d\sigma.
    \end{align*}
\end{proof}

\subsection{Better approximation}
Suppose a compact Einstein orbifold $(M_0,g_0)$ admits a resolution sequence, and its Einstein constant $\mu<0$. As we pointed out in Proposition \ref{Phi=0}, there exist parameters $0<t_i<\delta_i^2\ll1$ with $\lim_{i\to\infty}t_i=\lim_{i\to\infty}\delta_i=0$, such that for each gluing metric $\gtidi$, there is an Einstein metric $g_i$ with
\[\Phi_{\gtidi}(g_i-g_{t_i})=0,\quad \lim_{i\to\infty}\|g_i-g_{t_i}\|_{\ctab(g_{t_i})}=0.\]
For simplicity of notations, we omit the subscript $i$, denote $\gtd=\gtidi$, and $g=g_i$ as the Einstein metric in Bianchi gauge. Now we are going to perform some deformations on both $(M_0,g_0)$ and $(N,g_b)$, then get another gluing metric on $M=M_0\#N$. We will see that under proper deformations, the new gluing metric approximate the Einstein metric $g$ better.

Recall that in Propositions \ref{bdeform} and \ref{0deform}, for small $v\in\ker(d\Psi_0)$ and $w\in\ker(d\Psi'_0)$, we get deformations 
\[g_{b,v}:=g_b+h_{v},\text{ and }g_{0,w}:=g_0+h_{w}.\]
Also recall that we get a $(0,2)-$tensor field $h_2$ on $(N,g_b)$, in Proposition \ref{h2}. Now we glue $(N,g_{b,v}+t^2h_2)$ and $(M_0,g_{0,w})$ together, in the same manner as \eqref{naive-gluing}:
\begin{equation}\label{better-gluing}
    \gtdp:=t^2\chi_t(g_{b,v}+t^2h_2)+(1-\chi_t)g_{0,w}.
\end{equation}
Again, for simplicity of notations, we omit the dependence of $g_t'$ on $v,w$ here.
\begin{remark}
    Note that $h_2=O(r_b^2)$, and is not positive definite, hence $g_{b,v}+t^2h_2$ is not positive definite near infinity of $N$. But in the gluing construction, we truncate it in a precompact subset, so for $t$ small enough, $\gtdp$ defined above is indeed a metric.
\end{remark}
\begin{lemma}\label{lem4.12}
    There exists $\varepsilon>0$ such that, for each Einstein metric $g$ with
    \[\Phi(g-g_t)=0,\quad\|g-g_t\|_{\ctab(g_t)}<\varepsilon,\]
    there exist small $v\in\ker(d\Psi_0)$ and $w\in\ker(d\Psi'_0)$, such that the refined gluing metric $\gtdp$ satisfies
    \[g-\gtdp\in\tilde{\SO}(\gtd)^\perp.\]
\end{lemma}
\begin{proof}
    Recall that we have a projection $\pi$ onto $\TO(\gtd)$. Denote $V=\pi(g-\gtd)$, so $g-(\gtd+V)\in\TO(\gtd)^\perp$.

    \begin{figure}[h]
        \centering
        \includegraphics[scale=0.3]{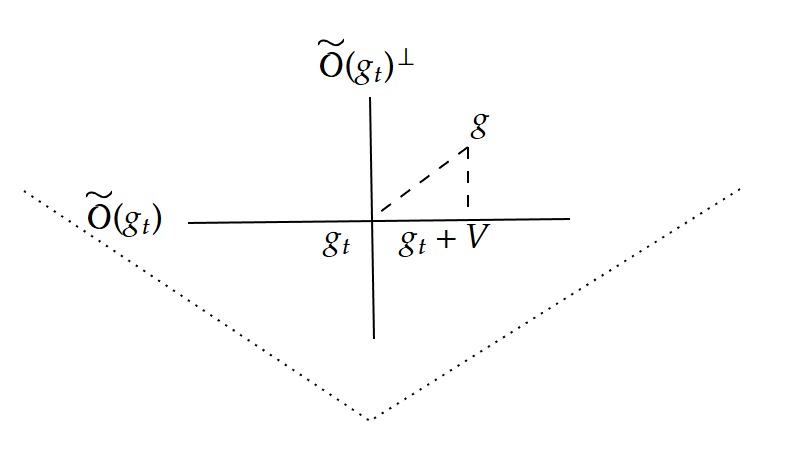}
        \caption{Cone of metrics on $M=M_0\#N$}
    \end{figure}

    Now it remains to prove that, we can find suitable $v\in\ker(d\Psi_0)$ and $w\in\ker(d\Psi'_0)$, such that $\gtdp-(\gtd+V)\in\TO(\gtd)^\perp$. This following from such a claim:

    \textbf{Claim:} The map $A:\ker(d\Psi_0)\oplus\ker(d\Psi'_0)\to\TO(\gtd)$, $(v,w)\mapsto \pi(\gtdp-\gtd)$ maps a neighbourhood of $0\in\ker(d\Psi_0)\oplus\ker(d\Psi'_0)$ surjectively to a neighbourhood of $0\in\TO(\gtd)$.

    For proof of this claim, it suffices to prove that $dA_0$ is surjective. By the "moreover" part of Lemma \ref{LSreduction}, we know both
    \[\left.\frac{dg_{b,v}}{dv}\right|_0:\ker(d\Psi_0)\to C^{2,\alpha}_\beta(g_b)\text{ and }\left.\frac{dg_{0,w}}{dw}\right|_0:\ker(d\Psi'_0)\to C^{2,\alpha}(g_0)\]
    are inclusions. Hence
    \[dA_0(v,w)=\pi\bigg(t^2\chi_tv+(1-\chi_t)w\bigg).\]\
    
    Since $\ker(d\Psi_0)=\SO(g_b)$, $\ker(d\Psi'_0)=\SO(g_0)\oplus\R\langle g_0\rangle$, the domain and target of $dA_0$ has the same dimension $1+\dim\SO(g_b)+\dim\SO(g_0)$. Now $dA_0$ has trivial kernel, hence $dA_0$ is surjective, and by implicit function theorem, the claim holds.
\end{proof}

\begin{proposition}\label{better-close}
    For Einstein metric $g$ with
    \[\Phi(g-g_t)=0,\quad\|g-g_t\|_{\ctab(g_t)}<\varepsilon,\]
    and refined gluing \( g_{t}' \) given by Lemma \ref{lem4.12}, there exists \( C>0 \) such that
    \[
    \| g - g_{t}' \|_{\ctab(g_{t})} \leq C \, t^{\frac{3-\beta}{2}}.
    \]
\end{proposition}
\begin{proof}
    \textbf{Step 1: Control by \( \pi^\perp\circ \Phi(g_{t}' - g_{t}) \).}

    We denote \( Q := \pi^\perp\circ \Phi - \pi^\perp\circ d\Phi_0 \) to be the nonlinear term of \( \pi^\perp\circ \Phi \), hence for \( x:=g - g_{t} \), \( y:=g_{t}' - g_{t} \),
    \[
    \| Q(x) - Q(y) \| \leq C \left( \|x\| + \|y\| \right) \| x - y \|.
    \]
    Since \( x-y = g - g_{t}' \in \TO(g_{t})^{\perp} \), by Corollary \ref{linear-estimate}, we have
    \begin{align*}
    \| g - g_{t}' \| &\leq C \| \pi^\perp\circ d\Phi_0(g - g_{t}') \| \\
    &= C \| \pi^\perp\circ d\Phi_0(x) - \pi^\perp\circ d\Phi_0(y) \| \\
    &\leq C \left( \| \pi^\perp\circ \Phi(x) - \pi^\perp\circ \Phi(y) \| + \| Q(x) - Q(y) \| \right) \\
    &\leq C \| \pi^\perp\circ \Phi(y) \| + C \left( \|x\| + \|y\| \right) \| x - y \|.
    \end{align*}
    By our choice of \( g \) and \( g_{t}' \), \( \|x\| \) and \( \|y\| \) are small enough, such that \( C(\|x\| + \|y\|)<1 \), hence
    \[
    \| g - g_{t}' \| \leq C \| \pi^\perp\circ \Phi(y) \| = C \| \pi^\perp\circ \Phi(g_{t}' - g_{t}) \|.
    \]

    \textbf{Step 2: Controlling \( \pi^\perp\circ \Phi(g_{t}' - g_{t}) \).}

    We recall the decomposition of \( M = M_0{\#}N \):
    \[
    N^\delta := \{ \rtd \leq t\delta^{-1} \}, \quad A(t,\delta) := \{ t\delta^{-1} \leq \rtd \leq \delta \}, \quad M_0^{\delta} := \{ \rtd \geq \delta \},
    \]
    and consider \( \Phi(g_{t}' - g_{t}) \) on each part.

    \begin{figure}[h]
        \centering
        \includegraphics[scale=0.3]{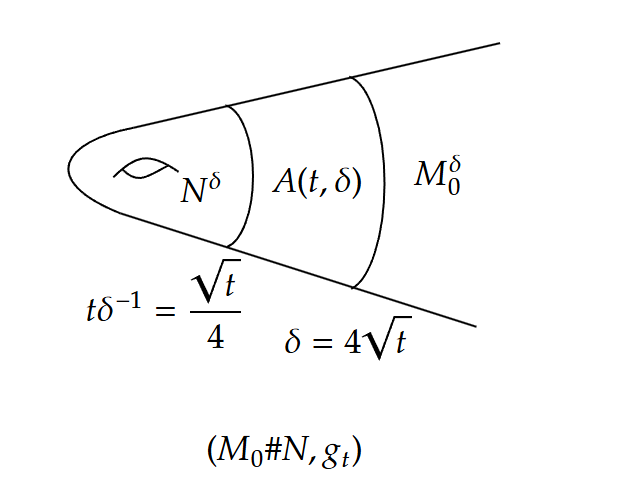}
        \caption{Different parts of the glued manifold}
    \end{figure}

    On \( N^\delta \), the main term is
    \begin{align*}
    \Phi\left( t^2(g_{b,v} + t^2h_2) - g_{t} \right) &= \operatorname{Ric}\left( t^2(g_{b,v} + t^2h_2) \right)-t^2\mu g_b + \delta_{g_{t}}^* B_{g_{t}}\left( t^2(g_{b,v} + t^2h_2) \right) + o(t^2) \\
    &= \operatorname{Ric}\left( t^2g_{b,v} \right) + \delta_{t^2g_b}^* B_{t^2g_b}\left( t^2g_{b,v} \right) + t^2P_{g_b}h_2-t^2\mu g_b+ o(t^2) \\
    &= {\Psi}(g_{b,v}-g_b) + t^2 \sum_j \lambda_j o_j + o(t^2),
    \end{align*}
    and by definition, \( {\Psi}(g_{b,v}-g_b), o_j \in \mathcal{O}(g_b) \).

    On \( M_0^{\delta} \), the main term is
    \[
    \Phi(g_{0,w} - g_{t}) = {\Psi}'(g_{0,w} - g_0) + O(t^2),
    \]
    and by definition, \( {\Psi}'(g_{0,w} - g_0) \in \mathcal{O}(g_0) \).

    On \( A(t,\delta) \), 
    \[ g_{t}' = \chi_t \left( t^2(g_{b,v} + t^2h_2) - g_{0,w} \right) + (1 - \chi_t) g_{0,w}, \]
    and we need to count the order of \( t^2(g_{b,v} + t^2h_2) - g_0 \). In suitable coordinates given by Lemma \ref{lemmaH},
    \[
    g_0 = g_E + H + O(r_{t}^3),
    \]
    and
    \[
    t^2(g_{b,v} + t^2h_2) = g_E + O(t^2r_b^{-m}) + H + O(t^4r_b^{2-m+\varepsilon}) = g_E + H + O\left(t^{m+2}\rtd^{-m}+t^{m+2-\varepsilon}\rtd^{2-m+\varepsilon}\right).
    \]
    Since on \( A(t,\delta) \), \( \rtd \sim \sqrt{t} \), we have
    \begin{align*}
    t^2(g_{b,v} + t^2h_2) - g_0 &= O(r_{t}^3) + O\left(t^{m+2}\rtd^{-m}\right)+O\left(t^{m+2-\varepsilon}\rtd^{2-m+\varepsilon}\right) \\
    &= O\left( t^{\frac{3}{2}} \right) + O\left(  t^{2+\frac{m}{2}} \right) + O\left( t^{3+\frac{m-\varepsilon}{2}} \right) \\
    &= O\left( t^{\frac{3}{2}} \right).
    \end{align*}
    Hence \( \rtd^{2} (\pi^\perp\circ\Phi)(g_{t}' - g_{t}) = O(t^{\frac{3}{2}}) \) on \( A(t,\delta) \), and
    \[
    \left\| \pi^\perp\circ \Phi(g_{t}' - g_{t})|_{A(t,\delta)} \right\|_{\cab(\gtd)} = \max\left\{ \rtd^\beta, \rtd^{-\beta} \right\} \cdot \left| \rtd^{2} (\pi^\perp\Phi)(g_{t}' - g_{t}) \right|_{C^{0,\alpha}} = O\left( t^{\frac{3 - \beta}{2}} \right).
    \]

    In summary, we have
    \[
    \left\| \pi^\perp\circ \Phi(g_{t}' - g_{t}) \right\|_{\cab(\gtd)} = O\left( t^{\frac{3 - \beta}{2}} \right).
    \]
\end{proof}

\section{Obstruction for resolution}
In this section, we analyze the obstruction for admitting a resolution sequence. By concentrating on the region \( N^{\delta} = \{x\in M_0\#N| \rtd(x) < t\delta^{-1} \}\cong\{x\in N|r_b(x)<\delta^{-1}\} \), the leading obstruction will be divided into two parts: the first part consists of integrability of the blow-up limit, and the second part is the obstruction of linearized Ricci equation in Section \ref{sec4.3}. 

For clarity, we consider the two parts separately. First we assume the blow-up limit is integrable, and treat the second part separately. Then we analysis the obstruction of integrability, using Bianchi identity.

Summing up, we are able to formulate an obstruction for general smooth blow-up limits. This leads to the proof of our main theorem \ref{Maintheorem}.

In the last subsection, we prove that compact hyperbolic orbifold with $\Z_2$ singularity does not admit a resolution sequence, and explain why the obstruction in \cite{MV} is a special case of our obstruction.
\subsection{For integrable blow-up limits}
Come back to our resolution problem. Given an Einstein orbifold \( (M_0, g_0) \) with a single orbifold point \( p \), and $\Ric(g_0)=\mu g_0$. Assume it admits a resolution sequence \( \{ (M_i, g_i) \} \subset \mathcal{M}(m, \Lambda, D, V, E) \), and the blow-up limit \( (N,g_b) \) at \( p \) is smooth. By previous discussions, we have:
\begin{itemize}
    \item[(1)] a sequence of direct gluing metrics \( g_{t_i} \) near \( g_i \), with \( \lim_{i \to \infty} t_i  = 0 \);
    \item[(2)] by gauge choice and slightly rescaling, we can assume
          \[
          B_{g_{t_i}} g_i = 0, \quad \operatorname{Ric}(g_i) = \mu g_i;
          \]
    \item[(3)] refined gluing metrics \( g_{t_i}' \), such that
          \[
          \| g_i - g_{t_i}' \|_{\ctab(\gtidi)} = O\left( t_i^{\frac{3 - \beta}{2}} \right).
          \]
\end{itemize}
Then we concentrate on \( (N^{\delta_i}, t_i^{-2}g_{t_i} = g_b) \). After rescaling, we rewrite (2) as:
\begin{itemize}
    \item[(2')] \( {\Psi}(t_i^{-2}g_i - g_b) = \operatorname{Ric}(t_i^{-2}g_i) + \delta_{g_b}^* B_{g_b}(t_i^{-2}g_i) = \mu g_i = t_i^2 \mu (t_i^{-2}g_i) \);
\end{itemize}
And from the proof of Proposition \ref{better-close}, on $N^{\delta_i}$ we have the estimate
\begin{itemize}
    \item[(3')] \( |t_i^{-2}g_i - (g_{b,v_i} + t_i^2 h_2)| = o\left( t_i^2 r_b^{-\beta} \right) \).
\end{itemize}

By Taylor's expansion, we have (denoting $h_{v_i}=g_{b,v_i}-g_b$ below, in the same convention as Proposition \ref{bdeform})
\begin{align*}
    0 &= \Psi(t_i^{-2}g_i - g_b) - t_i^2 \mu (t_i^{-2}g_i) \\
&= \Psi\left( g_{b,v_i} + t_i^2 h_2 - g_b \right) + o\left( t_i^2 r_b^{-2-\beta} \right) - t_i^2 \mu (t_i^{-2}g_i) \\
&= \Psi\left( h_{v_i} + t_i^2 h_2 \right) - t_i^2 \mu g_b +t_i^2\mu(g_b-t_i^{-2}g_i) + o\left( t_i^2 r_b^{-2-\beta} \right) \\
&= \Psi(h_{v_i}) + d\Psi_{h_{v_i}}(t_i^2 h_2) - t_i^2 \mu g_b+t_i^2\mu(g_b-t_i^{-2}g_i) + o\left( t_i^2 r_b^{-2-\beta} \right) \\
&= \Psi(h_{v_i})  + d\Psi_{0}(t_i^2 h_2) - t_i^2 \mu g_b+ \left(d\Psi_{h_{v_i}}-d\Psi_0\right)(t_i^2 h_2)+t_i^2\mu(g_b-t_i^{-2}g_i) + o\left( t_i^2 r_b^{-2-\beta} \right).
\end{align*}

Recall that by Proposition \ref{h2}, $(h_2,\lambda_j)$ is the solution to $d\Psi_{0}(h_2) - \mu g_b=\sum_{j}\lambda_jo_j$, where \( \{ o_j \} \) is an orthonormal basis of \( \mathcal{O}(g_b) \).

As pointed out in Proposition \ref{bdeform} and Remark \ref{rmk4.5}, for an arbitrary small $\varepsilon>0$, the map $C^{2,\alpha}_{m-2-\varepsilon}\cap\ker(P_{g_b})\to C^{2,\alpha}_{m-2-\varepsilon}$, $v\mapsto h_v$ is continuous. As $i\to\infty$, $v_i$ above tends to $0$, hence $\|h_{v_i}\|_{C^{2,\alpha}_{m-2-\varepsilon}}=o(1)$. Thus 
\[\left(d\Psi_{h_{v_i}}-d\Psi_0\right)(t_i^2 h_2)=o\left(t_i^2r_b^{2-m+\varepsilon}\right)=o\left( t_i^2 r_b^{-\beta} \right).\]

Also, by definition,
\[\left\|\left(t_i^{-2}g_i|_{N^{\delta_i}}\right)-\left(g_b|_{N^{\delta_i}}\right)\right\|_{C^{2,\alpha}_\beta(g_b)}\leq\|g_i-g_{t_i}\|_{C^{2\alpha}_{\beta,\beta}(g_{t_i})}=o(1),\]
so we have $t_i^2\mu(g_b-t_i^{-2}g_i)=o(t_i^2r_b^{-\beta})$. In sum, we have the following decomposition for leading obstruction:
\begin{equation}\label{decom-obstruction}
    \Psi(h_{v_i}) + t_i^2 \sum_j \lambda_j o_j = o\left( t_i^2 r_b^{-\beta} \right)\quad\text{on }N^{\delta_i}.
\end{equation}

Recall in Proposition \ref{bdeform}, we proved that for any \( v \in \mathcal{O}(g_b) \) with \( \|v\| \leq \eta \), there exists a unique \( h_v \in C^{2,\alpha}_\beta(g_b) \) in a neighborhood of 0, such that
\[
\Psi(h_v) \in \mathcal{O}(g_b), \quad h_v - v \in \mathcal{O}(g_b)^{\perp}.
\]
\begin{definition}
    We call a Ricci-flat ALE manifold \( (N, g_b) \) is \textbf{integrable}, if for each \( v \in \mathcal{O}(g_b) \) with \( \|v\| \leq \eta \), the tensor field \( h_v \) given by Proposition \ref{bdeform} satisfies \( \Psi(h_v) = 0 \).
\end{definition}

Suppose the blow-up limit \( (N, g_b) \) is integrable, then
\[
t_i^2 \sum_j \lambda_j o_j = o\left( t_i^2 r_b^{-\beta} \right) \quad \text{for the sequence } t_i \to 0,
\]
which forces each $\lambda_j=0$.
\begin{proposition}
    Given an Einstein orbifold \( (M_0, g_0) \) with a single orbifold point \( p \), and $\Ric(g_0)=\mu g_0$. Assume it admits a resolution sequence \( \{ (M_i, g_i) \} \subset \mathcal{M}(m, \Lambda, D, V, E) \), and the blow-up limit \( (N,g_b) \) at \( p \) is smooth and integrable. Then for any orthonormal basis $\{o_j\}$ of $\SO(g_b)$, we have each
    \[\lambda_j =-\frac{m+2}{2}\lim_{r \to \infty} \int_{S_r / \Gamma}  \frac{\langle H, o_j \rangle}{r}   d\sigma_{S_r / \Gamma}\]
    given by \eqref{lambda_j} vanishes, where $H=\left[- \frac{1}{3} R_{ikjl}(0) x^k x^l - \frac{2\mu}{3(m+2)}\left(|x|^2 \delta_{ij} + 2x^i x^j\right)\right]dx^i\otimes dx^j$ is given by \eqref{H}. 

    In particular, if we take $o_0\in\SO(g_b)$ to be proportional to $(\mathscr{L}_{\nabla u}g_b)^\circ$ given by \eqref{o0expansion}, then
    \begin{equation}\label{lambda_0}
        \lambda_0=-\frac{1}{\|(\mathscr{L}_{\nabla u}g_b)^\circ\|_{L^2(g_b)}}\left[2m(m-2)\mu\mathcal{V}+\frac{\omega_{m-1}}{3|\Gamma|}(W_{ikjl}(0)W^\infty_{ikjl}+W_{ikjl}(0)W^\infty_{iljk})\right]=0,
    \end{equation}
    where $\omega_{m-1}:=\mathrm{Area}_{g_E}(S^{m-1})$, and $W_{ijkl}(0)$ is the Weyl curvature of $(M_0,g_0)$ at $p$.
\end{proposition}
\begin{proof}
    We have seen that each $\lambda_j=0$ in the discussion above. It remains to compute $\lambda_0$ for $o_0=\frac{(\mathscr{L}_{\nabla u}g_b)^\circ}{\|(\mathscr{L}_{\nabla u}g_b)^\circ\|_{L^2(g_b)}}$. Then
    \[\lambda_0=-\frac{m+2}{2\|(\mathscr{L}_{\nabla u}g_b)^\circ\|_{L^2(g_b)}}\lim_{r \to \infty} \int_{S_r / \Gamma}  \frac{\langle H, (\mathscr{L}_{\nabla u}g_b)^\circ \rangle}{r}   d\sigma_{S_r / \Gamma}.\]
    Plug
    \[H_{ij}=- \frac{1}{3} R_{ikjl}(0) x^k x^l - \frac{2\mu}{3(m+2)}\left(|x|^2 \delta_{ij} + 2x^i x^j\right),\]
    \[(\mathscr{L}_{\nabla u}g)^{\circ}_{ij}=-4m\mathrm{Area}_{g_E}\left(S^{m-1}/\Gamma\right)^{-1}\mathcal{V}\left(\frac{mx^ix^j}{r^{m+2}}-\frac{\delta_{ij}}{r^m}\right)-2mW^\infty_{ikjl}\frac{x^kx^l}{r^{m+2}}+O(r^{-m-1+\varepsilon})\]
    into the above expression, we get
    \[-\frac{2}{m+2}\|(\mathscr{L}_{\nabla u}g_b)^\circ\|_{L^2(g_b)}\lambda_0=\frac{4m(m-2)}{m+2}\mu\mathcal{V}+\frac{2m}{3}R_{iajb}(0)W^\infty_{icjd}\int_{S^{m-1}/\Gamma}x^ax^bx^cx^dd\sigma_{S^{m-1}/\Gamma}.\]
    Using the formula 
    \[\int_{S^{m-1}}x^ax^bx^cx^dd\sigma_{S^{m-1}}=\frac{\omega_{m-1}}{m(m+2)}(\delta_{ab}\delta_{cd}+\delta_{ac}\delta_{bd}+\delta_{ad}\delta_{bc}),\]
    (see, for example, {\cite[Lemma 3.3]{MV}},) we finally get
    \[-\frac{2}{m+2}\|(\mathscr{L}_{\nabla u}g_b)^\circ\|_{L^2(g_b)}\lambda_0=\frac{4m(m-2)}{m+2}\mu\mathcal{V}+\frac{2\omega_{m-1}}{3(m+2)|\Gamma|}(R_{ikjl}(0)W^\infty_{ikjl}+R_{ikjl}(0)W^\infty_{iljk}).\]
    Finally, using curvature decomposition and $\Ric(g_0)=\mu g_0$, we have
    \[R_{ijkl}(0)=W_{ijkl}(0)+\frac{\mu}{m-1}(\delta_{ik}\delta_{jl}-\delta_{il}\delta_{jk}),\]
    which leads to our desired expression \eqref{lambda_0}.
\end{proof}

\subsection{Obstruction of integrability}
For a general smooth blow-up limit $(N,g_b)$, we have got in \eqref{decom-obstruction} that on $N^{\delta_i}$,
\[\Psi(h_{v_i})+t_i^2\sum_j\lambda_jo_j=o(t_i^2r_b^{-\beta}).\]
We need to treat $\Psi(h_{v_i})$ for non-integrable directions $v_i$.

For any small $v\in\ker(P_{g_b})=\SO(g_b)$, we consider a curve $\{h_{sv}\}_{s\in[0,1]}$ in $\Psi^{-1}(\coker(P_b))=\Psi^{-1}(\SO(g_b))$, and the derivatives $\left.\frac{d^k}{ds^k}\right|_{s=0}\Psi(h_{sv})$.
\begin{itemize}
    \item For $k=0$, we have $\Psi(0)=\Ric(g_b)+\delta^*_{g_b}B_{g_b}g_b=0$;
    \item For $k=1$, we have $\left.\frac{d}{ds}\right|_{s=0}\Psi(h_{sv})=P_{g_b}(v)=0$;
    \item For larger $k$, there are two cases:\begin{itemize}
        \item either $\left.\frac{d^k}{ds^k}\right|_{s=0}\Psi(h_{sv})=0$ for all $k\geq0$, then by analyticity of $\Psi$, we have $\Psi(h_{sv})\equiv0$;
        \item or there exists $l\geq2$, such that $\left.\frac{d^k}{ds^k} \right|_{s=0}\Psi(h_{sv}) = 0$ for all $k\leqslant l-1$ and $\left. \frac{d^l}{ds^l} \right|_{s=0}\Psi(h_{sv}) \neq  0$.
    \end{itemize}
    The first case is good enough for our application, as we have seen in Section 7.1. Now we treat the second case.
\end{itemize}
Before going ahead, remind our notation $g_{b,sv}=g_b+h_{sv}$.
\begin{lemma}
    Suppose $\left.\frac{d^k}{ds^k} \right|_{s=0}\Psi(h_{sv}) = 0$ for all $k\leqslant l-1$ and $\left. \frac{d^l}{ds^l} \right|_{s=0}\Psi(h_{sv}) \neq  0$.
    \begin{itemize}
    \item[(1)] We have
    \begin{gather}
         \left. \frac{d^{k-1}}{ds^{k-1}} \right|_{s=0} \Ric(g_{b,sv}) = 0, \tag{$A_k$}\\
         B_{g_b} \left( \left. \frac{d^k}{ds^k} \right|_{s=0} \Ric(g_{b,sv}) \right) = 0,\tag{$B_k$}\\
         B_{g_b} \left( \left. \frac{d^k}{ds^k} \right|_{s=0} g_{b,sv} \right) = 0,\tag{$C_k$}
    \end{gather}
    for all \( 1 \leq k \leq l \).

    \item[(2)] We have
    \begin{equation}\label{equ5.3}
        B_{g_b} \left( \left. \frac{d^l}{ds^l} \right|_{s=0} \Psi(h_{sv}) - d\Psi_0 \left( \left. \frac{d^l}{ds^l} \right|_{s=0} h_{sv} \right) \right) = 0.
    \end{equation}

    \item[(3)] We have
    \begin{equation}\label{equ5.4}
        \int_N\left\langle g_b,\left. \frac{d^l}{ds^l} \right|_{s=0} \Psi(h_{sv}) - d\Psi_0 \left( \left. \frac{d^l}{ds^l} \right|_{s=0} h_{sv} \right) \right\rangle_{g_b} dV_{g_b}=0.
    \end{equation}
\end{itemize}
\end{lemma}
\begin{proof}
    (1) We prove the equations by induction.

    \textbf{Step 1: Suppose \( (A_j) \) holds for all \( j \leq k \), we prove \( (B_k) \).}

    The contracted second Bianchi identity imply $B_{g_{b,sv}}\Ric(g_{b,sv}) = 0$, then we apply \( \left. \frac{d^k}{ds^k} \right|_{s=0} \) to this equation. By \( (A_j) \) for \( j \leq k \), we get:
    \[
        0 = \left. \frac{d^k}{ds^k} \right|_{s=0} B_{g_{b,sv}}\Ric(g_{b,sv}) = B_{g_b} \left( \left. \frac{d^k}{ds^k} \right|_{s=0} \Ric(g_{b,sv}) \right),
    \]
    which is \( (B_k) \).
    
    \textbf{Step 2: \( (B_k) \implies (C_k) \).}
    
    Since \( \Psi(h_{sv}) \in \mathcal{O}(g_b) \), we have \( B_{g_b} \Psi(h_{sv}) = 0 \) by Lemma \ref{IEDb}. Differentiate \( k \) times:
    \[
        0 = \left. \frac{d^k}{ds^k} \right|_{s=0} B_{g_b} \Psi(h_{sv}) = B_{g_b} \left( \left. \frac{d^k}{ds^k} \right|_{s=0}  \Ric(g_{b,sv}) + \delta^*_{g_b} B_{g_b}\left( \left. \frac{d^k}{ds^k} \right|_{s=0}g_{b,sv} \right) \right).
    \]
    By \( (B_k) \), the first term vanishes, hence \( B_{g_b} \delta^*_{g_b} B_{g_b} \left( \left. \frac{d^k}{ds^k} \right|_{s=0} g_s \right) = 0 \). By Ricci identity, \( B_{g_b} \delta^*_{g_b} =\nabla^*_{g_b}\nabla_{g_b}\). An integration by parts shows that \( (C_k) \) holds.
    
    \textbf{Step 3: \( (C_k) \implies (A_{k+1}) \) for \( k \leq l - 1 \).}
    
    By assumption:
    \[
        0 = \left. \frac{d^k}{ds^k} \right|_{s=0} \Psi(h_{sv}) = \left. \frac{d^k}{ds^k} \right|_{s=0} \Ric(g_{b,sv}) + \delta^*_{g_b} B_{g_b} \left( \left. \frac{d^k}{ds^k} \right|_{s=0} g_{b,sv} \right).
    \]
    By \( (C_k) \), the second term vanishes, so \( (A_{k+1}) \) holds.
    
    \textbf{Step 4: Induction process.}
    
    As a base case, \( \Ric(g_b) = 0 \), i.e. \( (A_1) \) holds. By Steps 1-3, \( (A_k), (B_k), (C_k) \) hold for \( 1 \leq k \leq l \). Note that step 3 needs $k\leq l - 1$, which is the only place we use the assumption. So the induction stops at \( k = l \).

    (2) For $B_{g_b} \left( \left. \frac{d^l}{ds^l} \right|_{s=0} \Psi(h_{sv}) - d\Psi_0 \left( \left. \frac{d^l}{ds^l} \right|_{s=0} h_{sv} \right) \right)$, we first note that
    $$\left. \frac{d^l}{ds^l} \right|_{s=0} \Psi(h_{sv}) - d\Psi_0 \left( \left. \frac{d^l}{ds^l} \right|_{s=0} h_{sv} \right)=\left. \frac{d^l}{ds^l} \right|_{s=0} \Ric(g_{b,sv}) - d\Ric_{g_b} \left( \left. \frac{d^l}{ds^l} \right|_{s=0} g_{b,sv} \right).$$
    By ($C_l$) in (1), we know that the first term lies in $\ker(B_{g_b})$. For the second term, we differentiate the Bianchi identity $B_{g_{b,sv}}\Ric(g_{b,sv})=0$ once, and get
    $$0=\left.\frac{d}{ds}\right|_{s=0}B_{g_{b,sv}}\Ric(g_{b,sv})=\left.\frac{d B_{g}}{dg}\right|_{g_b}\Ric(g_b)+B_{g_b}\circ d\Ric_{g_b}=B_{g_b}\circ d\Ric_{g_b}.$$

    (3) First note that
    \[\left. \frac{d^l}{ds^l} \right|_{s=0} \Psi(h_{sv}) - d\Psi_0 \left( \left. \frac{d^l}{ds^l} \right|_{s=0} h_{sv} \right)=\left. \frac{d^l}{ds^l} \right|_{s=0} \Ric(g_{b,sv}) - d\Ric_{g_b} \left( \left. \frac{d^l}{ds^l} \right|_{s=0} g_{b,sv} \right),\]
    so the desired integral is divided into two parts:
    \begin{align*}
        I=\int_N\left\langle g_b,\left. \frac{d^l}{ds^l} \right|_{s=0} \Ric(g_{b,sv})\right\rangle_{g_b}dV_{g_b},\\
        II=\int_N\left\langle g_b, d\Ric_{g_b} \left( \left. \frac{d^l}{ds^l} \right|_{s=0} g_{b,sv} \right)\right\rangle_{g_b}dV_{g_b}.
    \end{align*}

    For II, since $g_b$ is Ricci flat, $\left\langle g_b, d\Ric_{g_b} \left( \left. \frac{d^l}{ds^l} \right|_{s=0} g_{b,sv} \right)\right\rangle_{g_b}=d\,\mathrm{scal}_{g_b}\left(\left. \frac{d^l}{ds^l} \right|_{s=0} g_{b,sv}\right)$. By the variation formula of scalar curvature, and integration by parts,
    \[II=-\lim_{R\to\infty}\int_{r=R}(\delta_{g_b}+d\,\mathrm{tr}_{g_b})\left( \left. \frac{d^l}{ds^l} \right|_{s=0} g_{b,sv} \right)(\nu)d\sigma.\]

    For I, by $(A_k)$ in (1) for $1\le k\le l$, we get
    \[\left. \frac{d^l}{ds^l} \right|_{s=0}\mathrm{scal}_{g_{b,sv}}dV_{g_{b,sv}}=\left\langle g_b,\left. \frac{d^l}{ds^l} \right|_{s=0} \Ric(g_{b,sv})\right\rangle_{g_b}dV_{g_b},\]
    hence
    \[I=\left. \frac{d^l}{ds^l} \right|_{s=0}\int_N\mathrm{scal}_{g_{b,sv}}dV_{g_{b,sv}}.\]
    By the variation formula of the Einstein-Hilbert functional,
    \begin{align*}
        \frac{d}{ds}\int_N\mathrm{scal}_{g_{b,sv}}dV_{g_{b,sv}}=&-\int_N\left\langle\Ric(g_{b,sv})-\frac{\mathrm{scal}_{g_{b,sv}}}{2}g_{b,sv},\frac{d}{ds}g_{b,sv}\right\rangle_{g_{b,sv}}dV_{g_{b,sv}}\\
        &-\lim_{R\to\infty}\int_{r=R}(\delta_{g_{b,sv}}+d\,\mathrm{tr}_{g_{b,sv}})\left(\frac{d}{ds}g_{b,sv}\right)(\nu_s)d\sigma.
    \end{align*}
    Taking $l-1$ th derivative of $s$, and evaluate at $s=0$, the first term vanishes since it involves at most $l-1$ th derivatives of $\Ric(g_{b,sv})$. 

    In sum, the desired integral is
    \begin{align*}
        I-II=&\lim_{R\to\infty}\int_{r=R}(\delta_{g_b}+d\,\mathrm{tr}_{g_b})\left( \left. \frac{d^l}{ds^l} \right|_{s=0} g_{b,sv} \right)(\nu)d\sigma\\
        &-\lim_{R\to\infty}\left. \frac{d^{l-1}}{ds^{l-1}} \right|_{s=0}\int_{r=R}(\delta_{g_{b,sv}}+d\,\mathrm{tr}_{g_{b,sv}})\left(\frac{d}{ds}g_{b,sv}\right)(\nu_s)d\sigma.
    \end{align*}
    The boundary terms involves quadratic and higher degree polynomials of $\left. \frac{d^k}{ds^k} \right|_{s=0}h_{sv}$, and as pointed out in Remark \ref{rmk4.5}, $h_{sv}=O(r_b^{2-m+\varepsilon})$, hence the boundary terms vanish as $R\to\infty$.
\end{proof}
\begin{proposition}
    Suppose $\left.\frac{d^k}{ds^k} \right|_{s=0}\Psi(h_{sv}) = 0$ for all $k\leqslant l-1$ and $\left. \frac{d^l}{ds^l} \right|_{s=0}\Psi(h_{sv}) \neq  0$. Then
    \begin{equation}\label{equ5.5}
        \int_N\left\langle (\mathscr{L}_{\nabla u}g_b)^\circ,\left. \frac{d^l}{ds^l} \right|_{s=0} \Psi(h_{sv})  \right\rangle_{g_b} dV_{g_b}=0.
    \end{equation}
\end{proposition}
\begin{proof}
    For simplicity of notations, we denote
    \[S^l:=\left. \frac{d^l}{ds^l} \right|_{s=0} \Psi(h_{sv}) - d\Psi_0 \left( \left. \frac{d^l}{ds^l} \right|_{s=0} h_{sv} \right).\]
    Since $(\mathscr{L}_{\nabla u}g_b)^\circ\in\coker(P_{g_b})$, $d\Psi_0 \left( \left. \frac{d^l}{ds^l} \right|_{s=0} h_{sv} \right)\in\mathrm{Im}(P_{g_b})$. So the desired integral 
    \[\int_N\left\langle (\mathscr{L}_{\nabla u}g_b)^\circ,\left. \frac{d^l}{ds^l} \right|_{s=0} \Psi(h_{sv})  \right\rangle_{g_b} dV_{g_b}=\int_N\left\langle (\mathscr{L}_{\nabla u}g_b)^\circ, S^l  \right\rangle_{g_b} dV_{g_b}.\]
    By \eqref{equ5.3}, we know $\delta_{g_b}S^l=-\frac{1}{2}d\tr_{g_b}S^l$. By \eqref{equ5.4}, we have $\int_N\langle g_b,S^l\rangle_{g_b}dV_{g_b}=0$, hence
    \begin{align*}
        \int_N\left\langle (\mathscr{L}_{\nabla u}g_b)^\circ, S^l  \right\rangle_{g_b} dV_{g_b}&=\int_N\left\langle (\mathscr{L}_{\nabla u}g_b), S^l  \right\rangle_{g_b} dV_{g_b}\\
        &=2\int_N\left\langle \delta^*_{g_b}(\nabla u), S^l  \right\rangle_{g_b} dV_{g_b}\\
        &=2\lim_{R\to\infty}\int_{r=R}S^l(\nu,\nabla u)d\sigma+2\int_N\left\langle \nabla u, \delta_{g_b}S^l  \right\rangle_{g_b} dV_{g_b}\\
        &=-\int_N\langle \nabla u,d\tr_{g_b}S^l\rangle_{g_b}\\
        &=-\lim_{R\to\infty}\int_{r=R}(\nabla_{\nu}u)(\tr_{g_b}S^l)d\sigma+\int_N(\Delta_{g_b}u)(\tr_{g_b}S^l)dV_{g_b}\\
        &=2m\int_N\langle g_b,S^l\rangle_{g_b}dV_{g_b}\\
        &=0.
    \end{align*}
\end{proof}

\subsection{Proof of Theorem \ref{Maintheorem}}
Recall \eqref{decom-obstruction} that, on $N^{\delta_i}$, we have
\begin{equation}\label{equ5.6}
    \Psi(h_{v_i})+t_i^2\sum_j\lambda_jo_j=o(t_i^2r_b^{-\beta}).
\end{equation}
We write $\Psi(h_{v_i}) = V_i + O(\|v_i\|_{L^2(g_b)} \|V_i\|_{L^2(g_b)})$, where $V_i$ is the leading term
$\frac{1}{l!} \frac{d^l}{ds^l}\bigg|_{s=0} \Psi(h_{s v_i})$. After passing to a subsequence, we assume that for all $i$, the leading order $l$ in the Taylor expansion of $\Psi(h_{v_i})$ are the same. By \eqref{equ5.5}, we can take an orthonormal basis of $\SO(g_b)$ such that
\[
o_0 = \frac{(\mathscr{L}_{\nabla u}g_b)^\circ}{\|(\mathscr{L}_{\nabla u}g_b)^\circ\|_{L^2(g_b)}}, \quad o_1 = \frac{V_i}{\|V_i\|_{L^2(g_b)}}.
\]

\textbf{Step 1: Estimate $\|V_i\|_{L^2(g_b)}$.}

Take a cutoff function $\chi_i$ such that $\text{supp}(\chi_i) \subset N^{\delta_i} = \{ r_b \leq \frac{1}{4} t_i^{-\frac{1}{2}} \}$, and $\chi_i\equiv1$ on $\{r_b \leq \frac{1}{8} t_i^{-\frac{1}{2}}\}$. By \eqref{equ5.6}, we have
\begin{equation}\label{equ5.7}
\begin{aligned}
    0 &= \int_N \left\langle \Psi(h_{v_i}) + t_i^2 \sum_j \lambda_j o_j+o(t_i^2r_b^{-\beta}), \chi_i o_1 \right\rangle_{g_b} dV_{g_b}  \\
    &= \int_N \left\langle \Psi(h_{v_i}), \chi_i o_1 \right\rangle_{g_b} dV_{g_b} + t_i^2 \int_N \left\langle \sum_j \lambda_j o_j, \chi_i o_1 \right\rangle_{g_b} dV_{g_b} + o(t_i^{2})\\
    &=:A+B+o(t_i^2).
\end{aligned}
\end{equation}
Here we have used the pointwise decay $o_1=O(r_b^{-m})$ from Proposition \ref{IEDb}(3), to ensure that the integral of error term is indeed $o(t_i^2)$.

For the first term $A=\int_N \left\langle \Psi(h_{v_i}), \chi_i o_1 \right\rangle_{g_b} dV_{g_b}$, we have
\begin{align*}
A_1 &= \int_N \left\langle \Psi(h_{v_i}), o_1 \right\rangle_{g_b} dV_{g_b}\\
&= \int_N \left\langle V_i + O(\|v_i\|_{L^2(g_b)} \|V_i\|_{L^2(g_b)}), \frac{V_i}{\|V_i\|_{L^2(g_b)}} \right\rangle_{g_b} dV_{g_b} \\
&= \|V_i\|_{L^2(g_b)} + O(\|v_i\|_{L^2(g_b)} \|V_i\|_{L^2(g_b)}),
\end{align*}
and 
\begin{align*}
A_2 &= \int_N \left\langle \Psi(h_{v_i}), (1-\chi_i) o_1 \right\rangle_{g_b} dV_{g_b}  \\
&\leq C \int_{\{r_b \geq \frac{1}{8} t_i^{-\frac{1}{2}}\}} \left\langle \|V_i\|_{L^2(g_b)} o_1, o_1 \right\rangle_{g_b} dV_{g_b} \\
&= C \int_{\frac{1}{8} t_i^{-\frac{1}{2}}}^{+\infty} \|V_i\|_{L^2(g_b)} r^{-2m} r^{m-1} dr \\
&= O\left(t_i^{\frac{m}{2}} \|V_i\|_{L^2(g_b)}\right).
\end{align*}
Here we have used the pointwise decay $o_1=O(r_b^{-m})$ from Proposition \ref{IEDb}(3).

Subtract $A_2$ from $A_1$, we get
\[
A = A_1 - A_2 = \|V_i\|_{L^2(g_b)} + O(\|v_i\|_{L^2(g_b)} \|V_i\|_{L^2(g_b)}) + O\left(t_i^{\frac{m}{2}} \|V_i\|_{L^2(g_b)}\right).
\]

For the second term $B = t_i^2 \int_N \left\langle \sum_j \lambda_j o_j, \chi_i o_1 \right\rangle_{g_b} dV_{g_b}$, we have
\begin{align*}
B &= t_i^2 \int_N \left\langle \sum_j \lambda_j o_j, o_1 \right\rangle_{g_b} dV_{g_b} - t_i^2 \int_{\{r_b \geq \frac{1}{8} t_i^{-\frac{1}{2}}\}} \left\langle \sum_j \lambda_j o_j, (1-\chi_i) o_1 \right\rangle_{g_b} dV_{g_b} \\
&= t_i^2 \lambda_1 + o(t_i^2).
\end{align*}

Summing up, \eqref{equ5.7} is equivalent to
\[
o(t_i^2) = A + B = \|V_i\|_{L^2(g_b)} + O(\|v_i\|_{L^2(g_b)} \|V_i\|_{L^2(g_b)}) + O\left(t_i^{\frac{m}{2}} \|V_i\|_{L^2(g_b)}\right) + t_i^2 \lambda_1 + o(t_i^2).
\]
First of all, we can see $\|V_i\|_{L^2(g_b)} = o(1)$. Then
\[
\|V_i\|_{L^2(g_b)} = -t_i^2 \lambda_1 + o(\|V_i\|_{L^2(g_b)}) + o\left(t_i^{\frac{m}{2}}\right) + o(t_i^2),
\]
which implies $\|V_i\|_{L^2(g_b)} = O(t_i^2)$.

\textbf{Step 2: Pair \eqref{equ5.6} with $o_0$.}

Now we have $\Psi(h_{v_i}) = V_i + O(\|v_i\|_{L^2(g_b)} \|V_i\|_{L^2(g_b)}) = V_i + O(t_i^2 \|v_i\|_{L^2(g_b)})$. Then \eqref{equ5.6} reads as
\[V_i + t_i^2 \sum_j \lambda_j o_j = O(t_i^2 \|v_i\|_{L^2(g_b)}) + o(t_i^2 r_b^{-\beta}) \quad \text{on } \{ r_b \leq \frac{1}{4} t_i^{-\frac{1}{2}} \}.\]
Then
\begin{equation}\label{equ5.8}
\begin{aligned}
    \int_N \left\langle V_i + t_i^2 \sum_j \lambda_j o_j, \chi_i o_0 \right\rangle_{g_b} dV_{g_b} &= \int_N \left\langle O(t_i^2 \|v_i\|_{L^2(g_b)}) + o(t_i^2 r_b^{-\beta}), \chi_i o_0 \right\rangle_{g_b} dV_{g_b} \\
    &= O(t_i^2 \|v_i\|_{L^2(g_b)}) + o(t_i^2) \\
    &= o(t_i^2).
\end{aligned}
\end{equation}
Here we have used the pointwise decay $o_1=O(r_b^{-m})$ from Proposition \ref{IEDb}(3).

Again we divide the left-hand side into two parts: for
\[
A' = \int_N \left\langle V_i, \chi_i o_0 \right\rangle_{g_b} dV_{g_b},
\]
we compute $A_1' = \int_N \left\langle V_i, o_0 \right\rangle_{g_b} dV_{g_b} = 0$,
\[
A_2' = \int_N \left\langle V_i, (1-\chi_i) o_0 \right\rangle_{g_b} dV_{g_b} \leq C \int_{\frac{1}{8} t_i^{-\frac{1}{2}}}^{+\infty} \|V_i\|_{L^2(g_b)} r^{-2m} r^{m-1} dr = O\left(t_i^{2+\frac{m}{2}}\right).
\]
Hence $A = A_1' - A_2' = O\left(t_i^{2+\frac{m}{2}}\right)$.

For $B' = t_i^2 \int_N \left\langle \sum_j \lambda_j o_j, \chi_i o_0 \right\rangle_{g_b} dV_{g_b}$,
\begin{align*}
B' &= t_i^2 \int_N \left\langle \sum_j \lambda_j o_j, o_0 \right\rangle_{g_b} dV_{g_b} - t_i^2 \int_{\{r_b \geq \frac{1}{8} t_i^{-\frac{1}{2}}\}} \left\langle \sum_j \lambda_j o_j, (1-\chi_i) o_0 \right\rangle_{g_b} dV_{g_b} \\
&= t_i^2 \lambda_0 + o(t_i^2).
\end{align*}

Summing up, \eqref{equ5.8} is equivalent to 
\[
o(t_i^2) = A' + B' = O\left(t_i^{2+\frac{m}{2}}\right) + t_i^2 \lambda_0 + o(t_i^2),
\]
for the sequence $t_i \to 0$. This forces $\lambda_0 = 0$.

By the expression of $\lambda_0$ (computed in \eqref{lambda_0}), $\lambda_0=0$ is equivalent to
\[
\mu\mathcal{V}+\frac{\omega_{m-1}}{6m(m-2)|\Gamma|}(W_{ikjl}(0)W^\infty_{ikjl}+W_{ikjl}(0)W^\infty_{iljk})=0,
\]

So our main theorem holds when $(M_0,g_0)$ has a single orbifold point $p$, and blow-up limit at $p$ is smooth. In general case, when $(M,g_b)$ has more than one orbifold point, we perform the gluing construction of blow-up limit tree on every orbifold point. Then we concentrate on the singular point $p$ with smooth blow-up limit, and we get the same conclusion at $p$.

\subsection{Discussion on explicit examples}
\subsubsection{Hyperbolic orbifold with $\Z_2$-singularities}
Suppose $(M_0^m,g_0)$ is a compact hyperbolic orbifold with $\Z_2$-singularities. At each singular point modelled on $\R^m/\Z_2$, the blow-up limit is also asymptotic to $\mathbb{R}^m/{\mathbb{Z}_2}$.
\begin{proposition}
    Any non-flat Ricci-flat ALE orbifold $(N,g_b)$ asymptotic to $\mathbb{R}^m/{\mathbb{Z}_2}$ must be smooth.
\end{proposition}
\begin{proof}
    Suppose not, there is an orbifold point $p \in N$ modelled on $\mathbb{R}^m/{\Gamma}$ for some $|\Gamma| \geq 2$. Then
    \[
    \lim_{r \to 0} \frac{\text{Vol}_{g_b}(B(p,r))}{r^m} = \frac{\text{Vol}_{g_E}(B(0,1))}{|\Gamma|},
    \]
    \[
    \lim_{r \to \infty} \frac{\text{Vol}_{g_b}(B(p,r))}{r^m} = \frac{\text{Vol}_{g_E}(B(0,1))}{2}.
    \]
    Since $(N,g_b)$ is non-flat, we have a strict inequality in Bishop-Gromov's volume comparison:
    \[
    \lim_{r \to 0} \frac{\text{Vol}_{g_b}(B(p,r))}{r^m} > \lim_{r \to \infty} \frac{\text{Vol}_{g_b}(B(p,r))}{r^m},
    \]
    which implies $|\Gamma| < 2$, contradiction.
\end{proof}

\begin{proof}[Proof of Theorem \ref{sphere}]
    Suppose $(M_0^m,g_0)$ admits a resolution sequence in $\M(m,\Lambda,D,V,E)$, then the blow-up limit at each orbifold point is smooth, with renormalized volume $\mathcal{V}<0$. For $(M_0,g_0)$, we have $\mu=-(m-1)$ and $W(0)=0$, hence
    \[\mu\mathcal{V}+\frac{\omega_{m-1}}{6m(m-2)|\Gamma|}(W_{ikjl}(0)W^\infty_{ikjl}+W_{ikjl}(0)W^\infty_{iljk})=-(m-1)\mathcal{V}\neq 0,\]
    which contradicts with Theorem \ref{Maintheorem}.
\end{proof}

\subsubsection{The Calabi ansatz}
In \cite{MV}, the authors considered the special cases that $(N,g_b)=(O_{\CP^{n-1}}(-n),g_{Cal})$, $n>2$. They showed in {\cite[Theorem 2.1]{MV}} that $\SO(g_{Cal})$ is one-dimensional, hence it must be proportional to $(\mathscr{L}_{\nabla}g_b)^\circ$. Then they calculated
\[\lambda_0 =-\frac{2n+2}{2}\lim_{r \to \infty} \int_{S_r / \Gamma}  \frac{\langle H, o_0 \rangle}{r}   d\sigma_{S_r / \Gamma}\]
in methods of K\"ahler geometry, to get the obstruction {\cite[Theorem 1.1]{MV}}
\begin{equation}
    n\langle\mathcal{R}(\omega),\omega\rangle(p)+2(n-2)\mathrm{scal}(p)=0,
\end{equation}
where $\omega(p)$ is a canonically defined K\"ahler structure at $p$, given by the orbifold group $\Z_n$; and
\[\langle\mathcal{R}(\omega),\omega\rangle(p):=R_{ijkl}(p)\omega_{ij}(p)\omega_{kl}(p).\]

\bibliography{ref}
\vspace{0.5in}

Yichen Yao, School of Mathematical Sciences, University of Science and Technology of China, No. 96 Jinzhai Road, Hefei, Anhui Province, 230026, China; yichenyao@mail.ustc.edu.cn.\\
\end{document}